\documentclass[a4paper,10pt]{article}

\usepackage{geometry}
\usepackage[tbtags]{mathtools}
\usepackage{listings}
\usepackage{xcolor}
\usepackage{amsthm}
\usepackage{mathrsfs}
\usepackage{amssymb}
\usepackage{comment}

\usepackage{graphicx}
\usepackage{url}
\usepackage[numbers,sort&compress]{natbib}
\usepackage[hidelinks]{hyperref}
\allowdisplaybreaks[4]

\newtheorem{theorem}{Theorem}[section]
\newtheorem{lemma}[theorem]{Lemma}
\newtheorem{corollary}[theorem]{Corollary}
\newtheorem{proposition}[theorem]{Proposition}
\newtheorem{conjecture}[theorem]{Conjecture}
\newtheorem{problem}[theorem]{Problem}
\newtheorem{definition}[theorem]{Definition}
\newtheorem{remark}[theorem]{Remark}
\newtheorem{claim}{Claim}

\DeclareMathOperator{\Spec}{Spec}

\begin{document}

\title{\Large On Spielman's Laplacian Eigenratio Conjecture and Related Problems}

\author{
Jie Ma\footnote{School of Mathematical Sciences, University of Science and Technology of China, Hefei 230026, China.}
\footnote{Yau Mathematical Sciences Center, Tsinghua University, Beijing 100084, China.}
\and
Quanyu Tang\footnote{School of Mathematics and Statistics, Xi'an Jiaotong University, Xi'an 710049, China.}
\and
Yuchang Wang\footnotemark[1]
\and
Zhiheng Zheng\footnotemark[1]
}

\maketitle

\begin{abstract}
Let $G$ be an $n$-vertex graph with Laplacian eigenvalues $0=\lambda_1(G)\le \lambda_2(G)\le\cdots\le \lambda_n(G)$.
Motivated by the Alon--Boppana bound and the Ramanujan phenomenon for regular graphs, Spielman conjectured that, for every graph $G$ with fixed average degree $d\ge 1$, its {\it Laplacian eigenratio} satisfies
$$
\frac{\lambda_2(G)}{\lambda_n(G)}
\le
\frac{d-2\sqrt{d-1}}{d+2\sqrt{d-1}}+o_n(1),
$$
where $o_n(1)\to 0$ as $n\to\infty$.
The main purpose of this paper is to investigate this conjecture. We show that the situation is mixed. On the negative side, the conjecture fails for infinitely many average degrees $d>2$, via constructions based on bipartite Ramanujan graphs. On the positive side, it holds in two important settings: we verify it for all average degrees $d\le 2$, and we prove it for all regular graphs. In fact, for regular graphs we obtain stronger bounds comparing higher Laplacian eigenvalues. As a consequence, we show that for every fixed $d\ge 3$ and every $\varepsilon>0$, every sufficiently large $d$-regular Ramanujan graph has linearly many adjacency eigenvalues below $-2\sqrt{d-1}+\varepsilon$, thereby strengthening earlier results of Li and Cioab\u{a} by giving an unconditional result of this form. We also settle two related conjectures: one of You and Liu concerning the maximum Laplacian eigenratio of trees, and one of Gu concerning the Hamiltonicity of graphs with large Laplacian eigenratio.
\end{abstract}

\section{Introduction}

All graphs in this paper are finite and simple.
Let \(G=(V(G),E(G))\) be a graph with \(V(G)=\{v_1,\dots,v_n\}\).
Its \emph{adjacency matrix} is \(A_G=(a_{ij})_{n\times n}\), where \(a_{ij}=1\) if \(v_iv_j\in E(G)\) and \(a_{ij}=0\) otherwise.
The \emph{Laplacian matrix} of \(G\) is defined by
\(L_G=D_G-A_G\), where \(D_G\) is the diagonal matrix whose \((i,i)\)-entry is \(d_G(v_i)\), the degree of \(v_i\) in \(G\).
Throughout the paper, we write \(\mu_1(G)\ge \mu_2(G)\ge \cdots \ge \mu_n(G)\) for the eigenvalues of \(A_G\), and \(0=\lambda_1(G)\le \lambda_2(G)\le \cdots \le \lambda_n(G)\) for the eigenvalues of \(L_G\). A graph \(G\) is \emph{\(d\)-regular} if every vertex of \(G\) has degree \(d\); in this case, we have \(\lambda_i(G)=d-\mu_i(G)\) for each \(i=1,2,\dots,n\).

A central theme in spectral graph theory is the behavior of the extremal adjacency eigenvalues of large regular graphs.
The starting point is the celebrated Alon--Boppana bound (see Alon~\cite{alon1986eigenvalues} and Boppana~\cite{boppana1986eigenvalues}), which asserts that for every fixed integer \(d\ge 3\) and every sequence \(\{G_i\}\) of \(d\)-regular graphs with \(|V(G_i)|\to\infty\),
\[
\liminf_{i\to\infty}\mu_2(G_i)\ge 2\sqrt{d-1}.
\]
Thus \(2\sqrt{d-1}\) is the natural asymptotic barrier for the {\it nontrivial} adjacency spectrum of \(d\)-regular graphs (i.e., eigenvalues other than $\pm d$).
More generally, a theorem of Serre~\cite{serre1997repartition} implies that for every fixed integer \(s\ge 2\),
\[
\liminf_{i\to\infty}\mu_s(G_i)\ge 2\sqrt{d-1}.
\]
In fact, in an asymptotic sense, the parameter $s$ can be taken linear in the number of vertices.
Thus, the Alon--Boppana threshold governs not only the second eigenvalue but, quantitatively, a positive proportion of the top of the adjacency spectrum; see also~\cite{mohar2010strengthening}. Nilli~\cite{nilli1991second,nilli2004tight} later gave particularly elegant proofs and refinements of the Alon--Boppana bound, while Kahale~\cite{kahale1992second} obtained related improvements. 

This line of work naturally leads to the emergence of the conceptually important family of Ramanujan graphs.
A connected \(d\)-regular graph \(G\) is called \emph{Ramanujan} if every adjacency eigenvalue other than \(\pm d\) has absolute value at most \(2\sqrt{d-1}\).
In other words, Ramanujan graphs are precisely those regular graphs whose nontrivial spectrum is as small as the Alon--Boppana bound permits.
The first explicit constructions of infinite families of Ramanujan graphs were given by Lubotzky, Phillips, and Sarnak~\cite{lubotzky1988ramanujan} and independently by Margulis~\cite{margulis1988explicit}; Morgenstern~\cite{morgenstern1994existence} later extended these constructions to all degrees \(d\) such that \(d-1\) is a prime power.
A different manifestation of the same threshold appears in random regular graphs: Alon~\cite{alon1986eigenvalues} conjectured that a random \(d\)-regular graph has second eigenvalue at most \(2\sqrt{d-1}+o_n(1)\) asymptotically almost surely, and this was proved by Friedman~\cite{friedman2003proof}.

A major breakthrough of Marcus, Spielman, and Srivastava~\cite{MarcusSpielmanSrivastava2015} established the existence of bipartite \(d\)-regular Ramanujan graphs for every fixed \(d\ge 3\), later extended to all sizes~\cite{MarcusSpielmanSrivastava2018}. More recently, Huang, McKenzie, and Yau~\cite{huang2024ramanujan} proved that for every fixed \(d\ge 3\) and all sufficiently large \(N\), approximately \(69\%\) of \(d\)-regular \(N\)-vertex graphs are Ramanujan. This resolves a major problem and significantly deepens our understanding of the field.
For further background, we refer readers to the surveys Davidoff--Sarnak--Valette~\cite{davidoff2003elementary} and Hoory--Linial--Wigderson~\cite{hoory2006expander}.

Motivated by this circle of ideas, it is natural to ask whether there is a meaningful extension of the Alon--Boppana bound beyond the adjacency spectrum of regular graphs.
Our main motivation is a conjecture of Spielman~\cite[Section~27]{spielman2019spectral}, which seeks to control the \emph{Laplacian eigenratio}
\[
\frac{\lambda_2(G)}{\lambda_n(G)}
\]
in terms of the average degree of \(G\).\footnote{Strictly speaking, Spielman's original conjecture is formulated for weighted graphs, whereas in this paper we restrict attention to simple graphs, which we believe to be of primary interest. For the purpose of disproving the conjecture, however, this causes no loss of generality: since simple graphs form a subclass of weighted graphs, any counterexample in the simple-graph setting (as we show in Theorem~\ref{thm:counterexamples}) also disproves the original weighted conjecture.} Viewed from this perspective, Spielman's conjecture may be regarded as a Laplacian-eigenratio analogue of the Alon--Boppana bound.

\subsection{Spielman's conjecture and our results}
We begin by introducing Spielman's conjecture. The Alon--Boppana bound can be equivalently stated as follows: any $n$-vertex $d$-regular graph $G$ satisfies $\lambda_2(G) \le d - 2\sqrt{d-1} + o_n(1)$. Spielman observed that the same proof approach cannot be used to obtain the corresponding lower bound $\lambda_n(G) \ge d + 2\sqrt{d-1} + o_n(1)$ for such $G$, although this inequality does hold when $G$ is Ramanujan. In \cite{spielman2019spectral}, he asked whether a graph $G$ can approximate the complete graph any better than Ramanujan graphs do, and formulated the following conjecture.

\begin{conjecture}[Spielman~\cite{spielman2019spectral}]\label{conj:spielman}
For every fixed rational number \(d\ge 1\) and every fixed \(\varepsilon>0\), there exists an integer \(n_0=n_0(d,\varepsilon)\) such that every graph \(G\) of average degree \(d\) on \(n\ge n_0\) vertices satisfies
\[
\frac{\lambda_2(G)}{\lambda_n(G)}
\le
\frac{d-2\sqrt{d-1}}{d+2\sqrt{d-1}}+\varepsilon.
\]
\end{conjecture}

Our results show that the behavior of this conjecture is surprisingly mixed: it fails for infinitely many \(d>2\), holds for \(d \le 2\), but remains valid in the setting of regular graphs. We also obtain higher-multiplicity analogues for regular graphs, as well as an unconditional quantitative result for Ramanujan graphs.

We first construct counterexamples for infinitely many values of the average degree \(d>2\). The construction builds on the work of Marcus--Spielman--Srivastava~\cite{MarcusSpielmanSrivastava2015} on the existence of bipartite Ramanujan graphs.

\begin{theorem}\label{thm:counterexamples}
For each integer \(3\le q\le 9\), there exists an integer \(m_0(q)\) such that the following holds.
For every integer \(m\ge m_0(q)\), let
\[
d_{q,m}:=2+\frac{q-2}{m}.
\]
Then there exists a constant \(\varepsilon_{q,m}>0\) and infinitely many connected graphs \(G\) of average degree \(d_{q,m}\) such that
\[
\frac{\lambda_2(G)}{\lambda_n(G)}
>
\frac{d_{q,m}-2\sqrt{d_{q,m}-1}}{d_{q,m}+2\sqrt{d_{q,m}-1}}
+
\varepsilon_{q,m}.
\]
\end{theorem}

In particular, Conjecture~\ref{conj:spielman} fails for every such value \(d_{q,m}\). 
In contrast, our second result shows that the conjecture holds throughout the entire interval \(1\le d\le 2\). Moreover, in the borderline case \(d=2\), we obtain a quantitative estimate that is sharp up to a constant factor (see Remark~\ref{rem:unicyclic}).

\begin{theorem}\label{thm:d-at-most-2}
Conjecture~\ref{conj:spielman} holds for every fixed rational number \(1\le d\le 2\).
In particular, if \(G\) is a graph on \(n\ge 6\) vertices of average degree \(2\), then
\[
\frac{\lambda_2(G)}{\lambda_n(G)}\le \frac{9}{4n}.
\]
\end{theorem}

Although the conjecture fails in general for \(d>2\), it remains valid for all \(d\)-regular graphs with \(d\ge 2\). The next two theorems may be viewed as Laplacian eigenratio analogues and extensions of Nilli's approach \cite{nilli1991second,nilli2004tight} to the Alon--Boppana bound. The proof of the first theorem highlights the main idea that we adopt: we compare $\lambda_2$ and $\lambda_n$ by constructing two carefully chosen test vectors, which in a sense form a pair of counterparts, each supported near one of two distant edges.

\begin{theorem}\label{thm:regular-ratio}
Let \(d\ge 2\) and \(k\ge 1\) be integers, and let \(G\) be an \(n\)-vertex \(d\)-regular graph containing two edges at distance at least \(2k+2\).
Then
\[
\frac{\lambda_2(G)}{\lambda_n(G)}
<
\frac{d-2\sqrt{d-1}}{d+2\sqrt{d-1}}
+
\frac{4}{(k+1)\sqrt{d-1}}.
\]
\end{theorem}

Since every $n$-vertex \(d\)-regular graph with \(d\ge 3\) contains two edges at distance \(\Omega(\log_{d-1} n)\), Theorem~\ref{thm:regular-ratio} immediately implies Conjecture~\ref{conj:spielman} for all \(d\)-regular graphs with \(d\ge 3\); the case \(d=2\) follows from Theorem~\ref{thm:d-at-most-2}.

Moreover, in the setting of regular graphs, the proof idea of Theorem~\ref{thm:regular-ratio} yields stronger conclusions than Conjecture~\ref{conj:spielman}: we obtain inequalities comparing higher Laplacian eigenvalues, as well as a corresponding statement for the adjacency spectrum.

\begin{theorem}\label{thm:higher-eigs}
Let \(s,k\) be positive integers.
Then the following statements hold.
\begin{itemize}
    \item[(a)] Let \(d\ge 3\), and let \(G\) be an \(n\)-vertex \(d\)-regular graph containing \(4s-2\) edges that are pairwise at distance at least \(2k+2\).
    Then
    \[
    (d+2\sqrt{d-1})\lambda_{s+1}(G)
    \le
    (d-2\sqrt{d-1})\lambda_{n-s+1}(G)
    +
    \frac{4(d-1)^{3/2}}{k+1}.
    \]

    \item[(b)] Let \(d\ge 5\), and let \(G\) be an \(n\)-vertex \(d\)-regular graph containing \(2s+1\) vertices whose pairwise distances are all at least \(4k\).
    Then
    \[
    (\sqrt{d-1}+2)\mu_{s+1}(G)
    -
    (\sqrt{d-1}-2)\mu_{n-s+1}(G)
    \ge
    4(d-1)\cos\left(\frac{\pi}{2k}\right).
    \]
\end{itemize}
\end{theorem}

Using Theorem~\ref{thm:higher-eigs}, we also derive a new consequence for Ramanujan graphs. Serre's theorem~\cite{serre1997repartition} and its strengthenings (such as Friedman~\cite{friedman1993some}, Nilli~\cite{nilli2004tight} and Mohar~\cite{mohar2010strengthening}) imply that for every integer $d\geq 3$ and every $\varepsilon>0$, there exists a constant $\alpha=\alpha(d,\varepsilon)>0$ such that every \(n\)-vertex \(d\)-regular Ramanujan graph with \(n\) sufficiently large has at least \(\alpha n\) adjacency eigenvalues larger than \(2\sqrt{d-1}-\varepsilon\).
By contrast, the behavior near the negative extremal adjacency spectrum is generally more delicate, and the existing analogous results on negative spectrum of Ramanujan graphs often require additional assumptions, such as large odd girth; see, for instance, Li~\cite{li2001negative}, Cioab\u{a}~\cite{cioabua2006extreme}, and Mohar~\cite{mohar2010strengthening}.  
Our next theorem shows that Ramanujan graphs nevertheless have linearly many eigenvalues below $-2\sqrt{d-1} + \varepsilon$, without any assumption on odd girth or bipartiteness. 
This provides an unconditional result of this kind for general Ramanujan graphs.

\begin{theorem}\label{thm:ramanujan}
For every integer \(d\ge 3\) and every \(\varepsilon>0\), there exist \(n_0=n_0(d,\varepsilon)\) and \(\beta=\beta(d,\varepsilon)>0\) such that every \(n\)-vertex \(d\)-regular Ramanujan graph with \(n\ge n_0\) vertices has at least \(\beta n\) adjacency eigenvalues smaller than \(-2\sqrt{d-1}+\varepsilon\).
\end{theorem}

We present two proofs of this theorem: the first handles all \(d\ge 3\) using Theorem~\ref{thm:higher-eigs}\textup{(a)}, while the second, for \(d\ge 6\), refines the argument via Theorem~\ref{thm:higher-eigs}\textup{(b)} to yield a sharper quantitative estimate.

\subsection{Other problems related to the Laplacian eigenratio}

Besides Spielman's conjecture, we consider two further problems concerning the Laplacian eigenratio.

The first is a conjecture of You and Liu on trees. In~\cite[Conjecture~2.1]{YouLiu2012}, they conjectured that among all trees on \(n\) vertices, the star \(S_n\) maximizes \(\lambda_2(T)/\lambda_n(T)\), while the path minimizes it. The minimization half was later disproved by Lin, Cai and Wang~\cite{LinCaiWang2024}. The maximization half, however, remains open.

\begin{conjecture}[You--Liu~\cite{YouLiu2012}]\label{conj:you-liu}
Let \(T\) be a tree on \(n\ge 3\) vertices.
Then
\(
\frac{\lambda_2(T)}{\lambda_n(T)}
\le
\frac{\lambda_2(S_n)}{\lambda_n(S_n)}
=
\frac{1}{n}.
\)
Moreover, equality should hold if and only if \(T\cong S_n\).
\end{conjecture}

There has been partial progress on Conjecture~\ref{conj:you-liu}; see You and Liu~\cite{YouLiu2012} and Lin and Miao~\cite{LinMiao2021}.
Our next theorem settles it completely.

\begin{theorem}\label{thm:tree-ratio}
Let \(T\) be a tree on \(n\ge 3\) vertices.
Then
\(
\frac{\lambda_2(T)}{\lambda_n(T)}\le \frac1n,
\)
with equality if and only if \(T\cong S_n\).
\end{theorem}

The second problem concerns Hamiltonicity. Gu~\cite[Conjecture~5.11]{GH22} conjectured that a graph with sufficiently large Laplacian eigenratio must contain a Hamilton cycle; see also~\cite[Section~8]{LiuNing2023}.

\begin{conjecture}[Gu~\cite{GH22}]\label{Conj:GH}
There exists an absolute constant \(c<1\) such that for any graph \(G\) on \(n\ge 3\) vertices, if
\(
\frac{\lambda_2(G)}{\lambda_n(G)}\ge c,
\)
then \(G\) contains a Hamilton cycle.
\end{conjecture}

Our next theorem confirms this conjecture.

\begin{theorem}\label{thm:GHmain}
There exists an absolute constant \(c<1\) such that for any graph \(G\) on \(n\ge 3\) vertices, if
\(
\frac{\lambda_2(G)}{\lambda_n(G)}\ge c,
\)
then \(G\) contains a Hamilton cycle.
\end{theorem}

The proof uses Haemers' separation inequality (see, e.g.,~\cite{Haemers95}) to show that a large Laplacian eigenratio forces strong expansion, and then invokes the recent breakthrough of Dragani\'c, Montgomery, Correia, Pokrovskiy, and Sudakov~\cite{HamExpanders} on Hamiltonicity of graph expanders.

\subsection{Paper organization}

The rest of the paper is organized as follows.
In Section~\ref{sec:counterexamples}, we prove Theorem~\ref{thm:counterexamples} and show that Conjecture~\ref{conj:spielman} fails for infinitely many average degrees greater than \(2\).
Section~\ref{sec:low-degree} treats the regime \(1\le d\le 2\) and proves Theorem~\ref{thm:d-at-most-2}.
In Section~\ref{sec:regular}, we turn to the regular case and prove Theorems~\ref{thm:regular-ratio}--\ref{thm:ramanujan}.
Section~\ref{sec:trees} establishes Theorem~\ref{thm:tree-ratio}, thereby proving Conjecture~\ref{conj:you-liu}.
Section~\ref{sec:hamiltonicity} proves Theorem~\ref{thm:GHmain}, confirming Conjecture~\ref{Conj:GH}.
Finally, we conclude with a discussion of future directions.

\section{\texorpdfstring{Counterexamples for infinitely many average degrees above $2$}{Counterexamples for infinitely many degrees above 2}}\label{sec:counterexamples}

We begin with the negative result on Conjecture~\ref{conj:spielman}. 
In this section, we show that the conjecture fails for infinitely many average degrees slightly above $2$.

Our construction is based on bipartite Ramanujan graphs from Marcus--Spielman--Srivastava \cite{MarcusSpielmanSrivastava2015}. More precisely, starting from a connected bipartite $q$-regular Ramanujan graph $H$, we attach $m-1$ pendant vertices to each vertex of $H$. Equivalently, the resulting graph can be viewed as the \emph{corona} \(H\circ \overline{K_{m-1}}\). This terminology goes back to Frucht and Harary~\cite{FruchtHarary1970}, and general spectral formulae for corona graphs are available in the literature; see, for instance, \cite{BarikPatiSarma2007,Liu2014}. 
For convenience, and since we will need monotonicity in the base Laplacian eigenvalue, we include a short self-contained derivation for our setting as follows.
For any $\lambda\geq 0$, define the following \(m\times m\) matrix
\[
A_m(\lambda):=
\begin{pmatrix}
\lambda+m-1 & -1 & -1 & \cdots & -1\\
-1 & 1 & 0 & \cdots & 0\\
-1 & 0 & 1 & \cdots & 0\\
\vdots & \vdots & \vdots & \ddots & \vdots\\
-1 & 0 & 0 & \cdots & 1
\end{pmatrix}.
\]

\begin{lemma}\label{lem:pendant-block}
Let \(m\ge 2\) and \(\lambda\ge 0\). 
Then the spectrum of the matrix $A_m(\lambda)$ is given by 
\[
\Spec(A_m(\lambda))= \{1^{(m-2)},\theta_-(\lambda),\theta_+(\lambda)\},
\]  
\[
\text{where} \qquad  \theta_-(\lambda)
:=
\frac{m+\lambda-\sqrt{(m+\lambda)^2-4\lambda}}{2} \qquad \text{ and } \qquad \theta_+(\lambda)
:=
\frac{m+\lambda+\sqrt{(m+\lambda)^2-4\lambda}}{2}.
\]
In particular, both $\theta_-(\lambda)$ and $\theta_+(\lambda)$ are strictly increasing on $[0,\infty)$, and if $\lambda>0$, then \(0<\theta_-(\lambda)<1<\theta_+(\lambda)\).

\end{lemma}

\begin{proof}
Let
\(U:=\{(0,b_1,\dots,b_{m-1})\in\mathbb R^m: b_1+\cdots+b_{m-1}=0\}\), and
\(V:=\{(a,b,\dots,b)\in\mathbb R^m: a,b\in\mathbb R\}.\)
Then
\[
\mathbb R^m=U\oplus V,
\qquad
\text{where} \qquad \dim U=m-2 
\qquad \text{and} \qquad
\dim V=2.
\]

First, we note that both \(U\) and \(V\) are \(A_m(\lambda)\)-invariant. Indeed, if
\(\vec{u}=(0,b_1,\dots,b_{m-1})\in U\), then
\[
A_m(\lambda)\vec{u}
=
\bigl(-(b_1+\cdots+b_{m-1}),\,b_1,\dots,b_{m-1}\bigr)
=
(0,b_1,\dots,b_{m-1})=\vec{u},
\]
since \(b_1+\cdots+b_{m-1}=0\). Hence \(A_m(\lambda)|_U=I_U\), so \(1\) is an eigenvalue of \(A_m(\lambda)\) with multiplicity \(m-2\).
Now let \(\vec{v}=(a,b,\dots,b)\in V\). Then
\[
A_m(\lambda)\vec{v}
=
\bigl((\lambda+m-1)a-(m-1)b,\,-a+b,\dots,-a+b\bigr)\in V,
\]
so \(V\) is invariant. Relative to the basis consisting of the two vectors \((1,0,\dots,0)\) and \((0,1,\dots,1)\) of \(V\), 
the restriction \(A_m(\lambda)|_V\) is represented by
\[
\begin{pmatrix}
\lambda+m-1 & -(m-1)\\
-1 & 1
\end{pmatrix}.
\]
Therefore its two eigenvalues are the roots of
\[
\det
\begin{pmatrix}
\lambda+m-1-\theta & -(m-1)\\
-1 & 1-\theta
\end{pmatrix}
=0,
\]
namely
\[
\theta_-(\lambda)
=
\frac{m+\lambda-\sqrt{(m+\lambda)^2-4\lambda}}{2}\qquad \text{and} \qquad \theta_+(\lambda)
=
\frac{m+\lambda+\sqrt{(m+\lambda)^2-4\lambda}}{2}.
\]
It follows that
\[
\Spec(A_m(\lambda))
=
\{1^{(m-2)},\theta_-(\lambda),\theta_+(\lambda)\}.
\]

Consider the polynomial
\(p_\lambda(x):=x^2-(m+\lambda)x+\lambda,\)
whose roots are precisely \(\theta_-(\lambda)\) and \(\theta_+(\lambda)\). For $\lambda>0$, the polynomial \(p_\lambda(x)\) satisfies
\(p_\lambda(0)=\lambda>0\) and \(p_\lambda(1)=1-m<0,\)
so \(0<\theta_-(\lambda)<1<\theta_+(\lambda)\). 
Also, we have
\[
\theta_{-}'(\lambda)
=
\frac12\left(
1- \frac{m+\lambda-2}{\sqrt{(m+\lambda)^2-4\lambda}}
\right),\qquad \text{and} \qquad
\theta_{+}'(\lambda)
=
\frac12\left(
1+ \frac{m+\lambda-2}{\sqrt{(m+\lambda)^2-4\lambda}}
\right).
\]
Since
\((m+\lambda)^2-4\lambda-(m+\lambda-2)^2=4(m-1)>0,\)
we have \(\sqrt{(m+\lambda)^2-4\lambda}>|m+\lambda-2|\), and hence \(\theta_-'(\lambda)>0\) and \(\theta_+'(\lambda)>0\). 
In particular, both $\theta_-$ and $\theta_+$ are strictly increasing on $[0,\infty)$.
\end{proof}

Now we are ready to present the proof of Theorem~\ref{thm:counterexamples}.

\begin{proof}[Proof of Theorem~\ref{thm:counterexamples}]
Fix $3\leq q\leq 9$. By \cite[Theorem~5.5]{MarcusSpielmanSrivastava2015},
there exist infinitely many connected bipartite $q$-regular Ramanujan graphs.
Let $H$ be one such graph, and write $N:=|V(H)|$.

By enlarging \(m_0(q)\) if necessary, we may assume \(m_0(q)\ge 3\). Fix an integer \(m\ge m_0(q)\). Construct a graph $G=G_{q,m}(H)$ by attaching exactly $m-1$ pendant vertices to each vertex of $H$; equivalently, \(G = H\circ \overline{K_{m-1}}\). Then
$|V(G)|=N+(m-1)N=mN$,
$|E(G)|=\frac{qN}{2}+(m-1)N$,
so the average degree of $G$ is
\[
\overline d(G)=\frac{2|E(G)|}{|V(G)|}
=\frac{q+2(m-1)}{m}
=2+\frac{q-2}{m}
=d_{q,m}.
\]

We now compute the Laplacian spectrum of $G$. Label the vertices of $H$ as $v_1,\dots,v_N$. For each $r\in\{1,\dots,m-1\}$ and each
$i\in\{1,\dots,N\}$, let $u_i^{(r)}$ be the $r$th leaf attached to $v_i$.
Order the vertices of $G$ as
\[
v_1,\dots,v_N,\;
u_1^{(1)},\dots,u_N^{(1)},\;
u_1^{(2)},\dots,u_N^{(2)},\;
\dots,\;
u_1^{(m-1)},\dots,u_N^{(m-1)}.
\]
With this ordering, the Laplacian matrix \(L_G\) of \(G\) can be written as
\[
L_G=
\begin{pmatrix}
L_H+(m-1)I & -I & -I & \cdots & -I\\
-I & I & 0 & \cdots & 0\\
-I & 0 & I & \cdots & 0\\
\vdots & \vdots & \vdots & \ddots & \vdots\\
-I & 0 & 0 & \cdots & I
\end{pmatrix},
\]
where the rows and columns are partitioned into \(m\) blocks, each of size \(N\times N\).

Let
$0=\lambda_1(H)\le \lambda_2(H)\le \cdots \le \lambda_N(H)$
be the Laplacian eigenvalues of $H$, and let $\vec{x}_1,\dots,\vec{x}_N$ be an orthonormal
eigenbasis of $L_H$, so that
\[
L_H\vec{x}_i=\lambda_i(H)\vec{x}_i
\qquad (1\le i\le N).
\]
We regard all vectors as column vectors, and identify \(\mathbb R^{mN}\cong (\mathbb R^N)^{m}\) via block-column notation:
\[(\vec{y}_0,\dots,\vec{y}_{m-1}):=
\begin{pmatrix}
\vec{y}_0\\
\vec{y}_1\\
\vdots\\
\vec{y}_{m-1}
\end{pmatrix},
\qquad \vec{y}_0,\dots,\vec{y}_{m-1}\in \mathbb R^N.
\]
We equip $\mathbb R^{mN}$ with the standard inner product
\[
\langle (\vec{y}_0,\dots,\vec{y}_{m-1}),(\vec{z}_0,\dots,\vec{z}_{m-1})\rangle
:=
\sum_{r=0}^{m-1} \langle \vec{y}_r,\vec{z}_r\rangle_{\mathbb R^N}.
\]
For each $1\le i\le N$, define
\[
\mathcal W_i
:=
\Bigl\{
(a_0\vec{x}_i,a_1\vec{x}_i,\dots,a_{m-1}\vec{x}_i):a_0,\dots,a_{m-1}\in\mathbb R
\Bigr\}
\subseteq \mathbb R^{mN}.
\]
Thus $\mathcal W_i$ is the subspace of block vectors whose every block is a scalar
multiple of $\vec{x}_i$.

We claim that the subspaces $\mathcal W_1,\dots,\mathcal W_N$ are mutually orthogonal and that \(\mathbb R^{mN}=\mathcal W_1\oplus\cdots\oplus \mathcal W_N\). Indeed, for any vectors
\(\vec{u}=(a_0\vec{x}_i,\dots,a_{m-1}\vec{x}_i)\in \mathcal W_i\) and \(\vec{v}=(b_0\vec{x}_j,\dots,b_{m-1}\vec{x}_j)\in \mathcal W_j,\)
we have
\[
\langle \vec{u},\vec{v}\rangle
=
\sum_{r=0}^{m-1} \langle a_r\vec{x}_i,b_r\vec{x}_j\rangle
=
\sum_{r=0}^{m-1} a_rb_r\,\langle \vec{x}_i,\vec{x}_j\rangle.
\]
Hence $\langle \vec{u},\vec{v}\rangle=0$ whenever $i\ne j$, since
$\vec{x}_1,\dots,\vec{x}_N$ is an orthonormal basis of $\mathbb R^N$. Thus the subspaces
$\mathcal W_1,\dots,\mathcal W_N$ are pairwise orthogonal. Next, consider any
\(\vec{y}=(\vec{y}_0,\dots,\vec{y}_{m-1})\in \mathbb R^{mN}\),
where \(\vec{y}_r\in\mathbb R^N\) for each $0\leq r\leq m-1$.
Since $\vec{x}_1,\dots,\vec{x}_N$ is a basis of $\mathbb R^N$, for each $r$ we may write \(\vec{y}_r=\sum_{i=1}^N c_{r,i}\vec{x}_i\) for suitable scalars $c_{r,i}\in\mathbb R$. Therefore
\[
\vec{y}
=
\sum_{i=1}^N (c_{0,i}\vec{x}_i,c_{1,i}\vec{x}_i,\dots,c_{m-1,i}\vec{x}_i),
\]
and each summand belongs to $\mathcal W_i$. Hence \(\mathbb R^{mN}=\mathcal W_1+\cdots+\mathcal W_N\). Since these subspaces are pairwise orthogonal, the sum is direct, proving
\(
\mathbb R^{mN}=\mathcal W_1\oplus\cdots\oplus \mathcal W_N.
\)

Moreover, each $\mathcal W_i$ is $L_G$-invariant. Indeed, for any \((a_0\vec{x}_i,a_1\vec{x}_i,\dots,a_{m-1}\vec{x}_i)\in \mathcal W_i\), block multiplication with $L_G$ gives
\[
L_G(a_0\vec{x}_i,a_1\vec{x}_i,\dots,a_{m-1}\vec{x}_i)
=
\Bigl(
((\lambda_i(H)+m-1)a_0-a_1-\cdots-a_{m-1})\vec{x}_i,\,
(-a_0+a_1)\vec{x}_i,\,
\dots,\,
(-a_0+a_{m-1})\vec{x}_i
\Bigr).
\]
Since each block is again a scalar multiple of $\vec{x}_i$, the resulting vector still
belongs to $\mathcal W_i$. Therefore $\mathcal W_i$ is $L_G$-invariant. Relative to the basis
\[
(\vec{x}_i,0,\dots,0),\ (0,\vec{x}_i,0,\dots,0),\ \dots,\ (0,\dots,0,\vec{x}_i)
\]
of $\mathcal W_i$, the restriction of $L_G$ to $\mathcal W_i$ is represented by the matrix \(A_m(\lambda_i(H))\) from Lemma~\ref{lem:pendant-block}.

By Lemma~\ref{lem:pendant-block}, the eigenvalues contributed by \(\mathcal W_i\) are exactly
\[
1 \quad\text{with multiplicity }m-2,
\qquad
\theta_-(\lambda_i(H)),
\qquad
\theta_+(\lambda_i(H)).
\]
Because \(\mathbb R^{mN}=\mathcal W_1\oplus \cdots \oplus \mathcal W_N\), the full Laplacian spectrum of $G$ is
\[
\operatorname{Spec}(L_G)
=
\{1^{(N(m-2))}\}
\cup
\{\theta_-(\lambda_i(H)),\theta_+(\lambda_i(H)):1\le i\le N\}.
\]

Next we locate the second-smallest and largest eigenvalues of $L_G$. By Lemma~\ref{lem:pendant-block}, we know that both $\theta_-(\lambda)$ and $\theta_+(\lambda)$ are strictly increasing on $[0,\infty)$, and if $\lambda>0$, then \(0<\theta_-(\lambda)<1<\theta_+(\lambda)\). Because $H$ is connected, $\lambda_1(H)=0<\lambda_2(H)$. Also, $\theta_-(0)=0$ and $\theta_+(0)=m$. Since \(\theta_-(0)=0\), every other \(\theta_-(\lambda_i(H))\) lies in \((0,1)\), and all remaining eigenvalues are either \(1\) (when \(m\ge 3\)) or of the form \(\theta_+(\lambda_i(H))>1\), it follows that
\[
\lambda_2(G)=\theta_-(\lambda_2(H)),
\qquad
\lambda_{mN}(G)=\theta_+(\lambda_N(H)).
\]

Now use the special properties of $H$.
Let
\[
q=\mu_1(H)\ge \mu_2(H)\ge \cdots \ge \mu_N(H)=-q
\]
be the adjacency eigenvalues of $H$.
Since $H$ is $q$-regular, \(L_H=qI-A_H\), so
$\lambda_i(H)=q-\mu_i(H)$.
Because $H$ is bipartite, \(\lambda_N(H)=q-(-q)=2q\). Because $H$ is Ramanujan,
$\mu_2(H)\le 2\sqrt{q-1}$,
and therefore
\[
\lambda_2(H)=q-\mu_2(H)\ge q-2\sqrt{q-1}:=\alpha_q.
\]
By monotonicity of $\theta_-$ and $\theta_+$, we obtain $\lambda_2(G)\ge \theta_-(\alpha_q)$ and $\lambda_{mN}(G)=\theta_+(2q)$, and hence
\[
\frac{\lambda_2(G)}{\lambda_n(G)}
=
\frac{\lambda_2(G)}{\lambda_{mN}(G)}
\ge
\frac{\theta_-(\alpha_q)}{\theta_+(2q)}=:R_{q,m}.
\]

First, for $\theta_-(\alpha_q)$, rationalizing the numerator gives
\[
\theta_-(\alpha_q)
=
\frac{m+\alpha_q-\sqrt{(m+\alpha_q)^2-4\alpha_q}}{2}
=
\frac{2\alpha_q}{m+\alpha_q+\sqrt{(m+\alpha_q)^2-4\alpha_q}}.
\]
Now
\[
\sqrt{(m+\alpha_q)^2-4\alpha_q}
=
(m+\alpha_q)\sqrt{1-\frac{4\alpha_q}{(m+\alpha_q)^2}}
=
m+\alpha_q+O_q(m^{-1}),
\]
and hence \(
m+\alpha_q+\sqrt{(m+\alpha_q)^2-4\alpha_q}
=
2m+2\alpha_q+O_q(m^{-1})\). Therefore
\[
\theta_-(\alpha_q)
=
\frac{2\alpha_q}{2m+2\alpha_q+O_q(m^{-1})}
=
\frac{\alpha_q}{m}+O_q(m^{-2}).
\]
Next,
\[
\theta_+(2q)
=
\frac{m+2q+\sqrt{(m+2q)^2-8q}}{2}.
\]
Similarly,
\[
\sqrt{(m+2q)^2-8q}
=
(m+2q)\sqrt{1-\frac{8q}{(m+2q)^2}}
=
m+2q+O_q(m^{-1}),
\]
so \(
\theta_+(2q)
=
m+2q+O_q(m^{-1})\). It follows that
\[
R_{q,m}
=
\frac{\theta_-(\alpha_q)}{\theta_+(2q)}
=
\frac{\frac{\alpha_q}{m}+O_q(m^{-2})}{m+2q+O_q(m^{-1})}
=
\frac{\alpha_q}{m^2}+O_q(m^{-3}).
\]

We now compare $R_{q,m}$ with the conjectured bound at \(d_{q,m}=2+\frac{q-2}{m}\). Define
\[
B_{q,m}
:=
\frac{d_{q,m}-2\sqrt{d_{q,m}-1}}{d_{q,m}+2\sqrt{d_{q,m}-1}}.
\]
Writing \(\delta:=\frac{q-2}{m}\), we have
\[
B_{q,m}
=
\frac{2+\delta-2\sqrt{1+\delta}}{2+\delta+2\sqrt{1+\delta}}.
\]
Using the Taylor expansion \(
\sqrt{1+\delta}
=
1+\frac{\delta}{2}-\frac{\delta^2}{8}+O(\delta^3)\), we obtain
\[
2+\delta-2\sqrt{1+\delta}
=
2+\delta-2\left(1+\frac{\delta}{2}-\frac{\delta^2}{8}+O(\delta^3)\right)
=
\frac{\delta^2}{4}+O(\delta^3),
\]
and
\[
2+\delta+2\sqrt{1+\delta}
=
2+\delta+2\left(1+\frac{\delta}{2}-\frac{\delta^2}{8}+O(\delta^3)\right)
=
4+2\delta-\frac{\delta^2}{4}+O(\delta^3)
=
4+O(\delta).
\]
Therefore
\[
B_{q,m}
=
\frac{\frac{\delta^2}{4}+O(\delta^3)}{4+O(\delta)}
=
\frac{\delta^2}{16}+O(\delta^3)
=
\frac{(q-2)^2}{16m^2}+O_q(m^{-3}).
\]

Combining the two asymptotic expansions, we conclude that
\[
R_{q,m}-B_{q,m}
=
\left(
\alpha_q-\frac{(q-2)^2}{16}
\right)\frac{1}{m^2}
+O_q(m^{-3}).
\]
Finally,
\[
\alpha_q-\frac{(q-2)^2}{16}>0
\iff
q-2\sqrt{q-1}>\frac{(q-2)^2}{16}
\iff
(\sqrt{q-1}+1)^2<16
\iff
q<10.
\]
Since $q\leq 9$, the coefficient of $m^{-2}$ is positive. Therefore
there exists an integer $m_0(q)$ such that for all $m\ge m_0(q)$,
\[
R_{q,m}>B_{q,m}.
\]

For fixed $(q,m)$ with $m\ge m_0(q)$, define
\[
\varepsilon_{q,m}:=\frac{R_{q,m}-B_{q,m}}{2}>0.
\]
Then every graph $G$ produced by the above construction satisfies
\[
\frac{\lambda_2(G)}{\lambda_n(G)}
\ge R_{q,m}
>
B_{q,m}+\varepsilon_{q,m},
\]
and since there are infinitely many such graphs $G$, Conjecture~\ref{conj:spielman}
fails for $d=d_{q,m}$.
\end{proof}

\section{\texorpdfstring{The cases $d\leq 2$}{The cases d<=2}}\label{sec:low-degree}

We now turn to the positive side below the threshold $d=2$. In this section we show that Conjecture~\ref{conj:spielman} holds for every fixed average degree $1\le d\le 2$. We first dispose of the easy range $1\le d<2$, and then treat the borderline case $d=2$.

\begin{proposition}\label{prop:d-less-than-2}
Let $1\le d<2$. Then every graph $G$ on $n>\frac{2}{2-d}$ vertices with average degree $d$ is disconnected. Consequently,
\[
\frac{\lambda_2(G)}{\lambda_n(G)}=0
\le
\frac{d-2\sqrt{d-1}}{d+2\sqrt{d-1}}
\]
for every such graph $G$.
\end{proposition}

\begin{proof}
Let $G$ be a graph on $n>\frac{2}{2-d}$ vertices with average degree $d$, and suppose that $G$ is connected. Then $|E(G)|\ge n-1$, and therefore
$d=\frac{2|E(G)|}{n}\ge \frac{2(n-1)}{n}=2-\frac{2}{n}>d$,
a contradiction. Hence $G$ is disconnected, so $\lambda_2(G)=0$.
Finally, for $1\le d<2$ one has
$\frac{d-2\sqrt{d-1}}{d+2\sqrt{d-1}}
=\left(\frac{1-\sqrt{d-1}}{1+\sqrt{d-1}}\right)^2\ge 0$,
which proves the claim.
\end{proof}

The remaining case $d=2$ is more intricate. Once the graph is connected, it is necessarily unicyclic, and we can exploit that structure.
For a graph $H$ and a vertex set $S\subseteq V(H)$, we write
\[
\partial_H S:=E_H\bigl(S,V(H)\setminus S\bigr)
\]
for the edge boundary of $S$. When the ambient graph is clear, we simply write $\partial S$.

The first ingredient is a standard cut bound for $\lambda_2$.

\begin{lemma}\label{lem:cut}
Let $G$ be a graph on $n$ vertices, and let $S\subseteq V(G)$ satisfy
$1\le |S|\le n-1$. Then
\[
\lambda_2(G)\le \frac{n|\partial S|}{|S|(n-|S|)}.
\]
\end{lemma}

\begin{proof}
Let $s=|S|$ and $t=n-s$. Define $\vec{x}\in \mathbb{R}^{V(G)}$ by
\[
x_u=
\begin{cases}
 t,&u\in S,\\
 -s,&u\notin S.
\end{cases}
\]
Then $\sum_{u\in V(G)}x_u=st-ts=0$, so $\vec{x}\perp \vec{1}$. Also,
\[
\vec{x}^{\top}L_G\vec{x}=\sum_{\{u,v\}\in E(G)}(x_u-x_v)^2=n^2|\partial S|,
\]
because only edges in $\partial S$ contribute, and each such edge contributes
$n^2$. Moreover,
\[
\vec{x}^{\top}\vec{x}=st^2+ts^2=nst.
\]
Hence the Rayleigh quotient of $\vec{x}$ is
\[
\frac{\vec{x}^{\top}L_G\vec{x}}{\vec{x}^{\top}\vec{x}}
=
\frac{n|\partial S|}{s(n-s)}.
\]
By the variational characterization of $\lambda_2(G)$, the result follows.
\end{proof}

To use this effectively, we also need a lower bound on the top Laplacian eigenvalue. The next lemma is standard; see~\cite[Theorem~4.12]{Bapat2010}.

\begin{lemma}\label{lem:lambdan-delta}
If $G$ is connected and has maximum degree $\Delta$, then \(\lambda_n(G)\ge \Delta+1\).
\end{lemma}

\begin{corollary}\label{cor:unicyclic-lambdan}
If $G$ is connected and unicyclic, then \(\lambda_n(G)\ge 3\).
\end{corollary}

\begin{proof}
A connected unicyclic graph has maximum degree at least $2$, so by Lemma~\ref{lem:lambdan-delta}, we have \(\lambda_n(G)\ge \Delta(G)+1\ge 3\).
\end{proof}

To handle the complementary case of bounded maximum degree, we need two simple structural lemmas about the unique cycle and its pendant trees.

\begin{lemma}\label{lem:cyclic-interval}
Let $w_1,\dots,w_r>0$ be cyclically ordered numbers such that $\sum_{i=1}^r w_i=n$ and $\max_{1\le i\le r} w_i<\frac n2$. Then there exists a cyclic interval $I\subseteq \{1,\dots,r\}$ such that
\[
\frac n3\le \sum_{i\in I} w_i\le \frac n2.
\]
\end{lemma}

\begin{proof}
After a cyclic relabeling, let $j$ be the smallest index such that
\[
s_j:=\sum_{i=1}^j w_i\ge \frac n3.
\]
If $s_j\le \frac n2$, then $I=\{1,\dots,j\}$ works. Assume next that $\frac n2<s_j\le \frac{2n}{3}$. Then the complementary cyclic
interval $I=\{j+1,\dots,r\}$ satisfies
\[
\frac n3\le n-s_j<\frac n2.
\]
Finally, assume that $s_j>\frac{2n}{3}$. Since $j$ was chosen minimally, we
have $s_{j-1}<\frac n3$, and hence
\[
w_j=s_j-s_{j-1}>\frac{2n}{3}-\frac n3=\frac n3.
\]
Because $w_j<\frac n2$ by hypothesis, the one-term interval $I=\{j\}$ works.
This proves the lemma.
\end{proof}

\begin{lemma}\label{lem:large-tree}
Let $G$ be a connected unicyclic graph on $n$ vertices, and let $C=v_1v_2\cdots v_rv_1$ be its unique cycle. For each $i\in\{1,\dots,r\}$, let $T_i$ be the component of $G-E(C)$ containing $v_i$. Assume that there exists $j_0$ such that \(|V(T_{j_0})|\ge \frac n2\). Then
\[
\frac{\lambda_2(G)}{\lambda_n(G)}\le \frac{2}{n}.
\]
\end{lemma}

\begin{proof}
Clearly we have $n\geq 3$.
If $G$ is a cycle, then $|V(T_i)|=1$ for every $i$,
which contradicts the assumption
$|V(T_{j_0})|\ge \lceil n/2\rceil \geq 2$. 
So in the rest of the proof, we assume that $G$ is not a cycle.

Root $T_{j_0}$ at $v_{j_0}$. Starting from the root, move downward as long as the current vertex has a child subtree of size strictly larger than $n/2$. Since subtree sizes strictly decrease at each such step, this process must terminate. Let $u$ be the terminal vertex, and let \(C_1,\dots,C_t\) be the child subtrees of $u$ in the rooted tree $T_{j_0}$. Relabel so that
\[
|V(C_1)|\ge \cdots \ge |V(C_t)|.
\]
Set \(a:=|V(C_1)|\). By construction, $a\le n/2$. Let $S:=V(C_1)$. Since $C_1$ is a child subtree of $u$, it meets $V(G)\setminus S$ in exactly one edge, namely the parent edge joining the root of $C_1$ to $u$. Hence \(|\partial_G S|=1\). By Lemma~\ref{lem:cut},
\[
\lambda_2(G)\le \frac{n}{a(n-a)}.
\]

We now distinguish two cases.

\smallskip
\noindent\underline{\textit{Case 1: $u\neq v_{j_0}$.}}
Then the rooted subtree $T_u$ has more than $n/2$ vertices, because we moved
from the parent of $u$ to $u$ only when the child subtree at $u$ had size
$>n/2$. Therefore
\[
\sum_{\ell=1}^t |V(C_\ell)|=|V(T_u)|-1\ge \left\lfloor \frac n2\right\rfloor, \qquad \text{ implying that } \qquad
a\ge \frac{\lfloor n/2\rfloor}{t}.
\]
Also, $u$ has exactly $t$ children and one parent in $G$, so \(\deg_G(u)=t+1\). Then Lemma~\ref{lem:lambdan-delta} gives
\[
\lambda_n(G)\ge \deg_G(u)+1=t+2, \qquad \text{ and thus } \qquad
\frac{\lambda_2(G)}{\lambda_n(G)}
\le
\frac{n}{a(n-a)(t+2)}.
\]
If $n=2m$ is even, then $m/t \le a \le n/2$, so
\[
\frac{\lambda_2(G)}{\lambda_n(G)}
\le
\frac{2m}{(m/t)(2m-m/t)(t+2)}
=
\frac{4t^2}{(2t-1)(t+2)}\cdot \frac1n
\le \frac{2}{n}.
\]
If $n=2m+1$ is odd, then $m/t \le a \le n/2$, and therefore
\[
\frac{\lambda_2(G)}{\lambda_n(G)}
\le
\frac{2m+1}{(m/t)(2m+1-m/t)(t+2)}.
\]
So it is enough to show that
\[
2(t+2)\frac{m}{t}\left(2m+1-\frac{m}{t}\right)\ge (2m+1)^2.
\]
After multiplying by $t^2$, this becomes
\[
6m^2t-4m^2-2mt^2+4mt-t^2\ge 0.
\]
The left-hand side is a concave quadratic in $t$. Since $1\le t\le n-2=2m-1$,
it suffices to check the endpoints $t=1$ and $t=2m-1$. At these endpoints it
equals
\[
2m^2+2m-1
\qquad\text{and}\qquad
(2m+1)(2m^2-1),
\]
respectively, both positive. Hence $\lambda_2(G)/\lambda_n(G)\le 2/n$.

\medskip

\noindent\underline{\textit{Case 2: $u=v_{j_0}$.}}
In this case,
\[
\sum_{\ell=1}^t |V(C_\ell)|=|V(T_{j_0})|-1\ge \frac n2-1,
\]
and hence
\[
a\ge \frac{n/2-1}{t}=\frac{n-2}{2t}.
\]
Also, $u=v_{j_0}$ has $t$ children and two neighbours on the cycle, so \(\deg_G(u)=t+2\). Therefore Lemma~\ref{lem:lambdan-delta} gives
\[
\lambda_n(G)\ge \deg_G(u)+1=t+3.
\]
Consequently,
\[
\frac{\lambda_2(G)}{\lambda_n(G)}
\le
\frac{n}{a(n-a)(t+3)}
\le
\frac{n}{\bigl((n-2)/(2t)\bigr)\bigl(n-(n-2)/(2t)\bigr)(t+3)}.
\]
Thus it remains to show that
\[
(t+3)(n-2)(2tn-n+2)\ge 2t^2n^2.
\]
The left-hand side minus the right-hand side equals
\[
5n^2t-3n^2-4nt^2-8nt+12n-4t-12.
\]
Again this is a concave quadratic in $t$. Since $1\le t\le n-3$, it is enough
to check the endpoints $t=1$ and $t=n-3$. At these endpoints it equals
\[
2(n^2-8)
\qquad\text{and}\qquad
n(n^2-2n-4),
\]
respectively, both positive for $n\ge 4$. Hence
\[
\frac{\lambda_2(G)}{\lambda_n(G)}\le \frac{2}{n}
\]
also in this case. This completes the proof.
\end{proof}

Equipped with these lemmas, we are ready to prove Theorem~\ref{thm:d-at-most-2}.

\begin{proof}[Proof of Theorem~\ref{thm:d-at-most-2}]
If \(1\le d<2\), then the proof follows immediately from Proposition~\ref{prop:d-less-than-2}. Therefore, it remains only to consider the case \(d=2\).

If $G$ is disconnected, then $\lambda_2(G)=0$, and there is nothing to prove.
Assume from now on that $G$ is connected. Since the average degree of $G$ is
$2$, we have $|E(G)|=|V(G)|=n$, so $G$ is unicyclic.
If $G\cong C_n$, then
\[
\lambda_2(G)=2-2\cos\left(\frac{2\pi}{n}\right)\le \frac{4\pi^2}{n^2},
\]
while Corollary~\ref{cor:unicyclic-lambdan} gives $\lambda_n(G)\ge 3$. Then, since $n\ge 6$, we have
\[
\frac{\lambda_2(G)}{\lambda_n(G)}
\le
\frac{4\pi^2}{3n^2}
<
\frac{9}{4n}.
\]

Thus we may assume that $G$ is not a cycle. Let $C=v_1v_2\cdots v_rv_1$ be the unique cycle of $G$. For each $i\in\{1,\dots,r\}$, let $T_i$ be the component of $G-E(C)$ containing $v_i$, and write \(w_i:=|V(T_i)|\). Since $G$ is not a cycle, we have $\Delta(G)\ge 3$, and Lemma~\ref{lem:lambdan-delta} gives
\begin{equation}\label{equ:Delta}
    \lambda_n(G)\ge \Delta(G)+1\ge 4.
\end{equation}
If $\max_i w_i\ge n/2$, then by Lemma~\ref{lem:large-tree}, we can derive
\[
\frac{\lambda_2(G)}{\lambda_n(G)}\le \frac{2}{n}<\frac{9}{4n},
\]
as desired.
Hence, it suffices to consider the case when $\max_i w_i<n/2$. 
By Lemma~\ref{lem:cyclic-interval}, there exists a cyclic interval
$I\subseteq\{1,\dots,r\}$ such that
\[
\frac n3\le \sum_{i\in I} w_i\le \frac n2.
\]
Set \(S:=\bigcup_{i\in I}V(T_i)\). Since $I$ is proper, the only edges between $S$ and $V(G)\setminus S$ are the two cycle edges at the ends of the interval, so $|\partial_G S|=2$. Applying Lemma~\ref{lem:cut}, we obtain
\[
\lambda_2(G)
\le
\frac{n|\partial_G S|}{|S|(n-|S|)}
\le
\frac{2n}{(n/3)(2n/3)}
=
\frac{9}{n}.
\]
Therefore, using \eqref{equ:Delta} we can derive that
\[
\frac{\lambda_2(G)}{\lambda_n(G)}
\le
\frac{9/n}{4}
=
\frac{9}{4n}.
\]
This completes the proof of Theorem~\ref{thm:d-at-most-2}.
\end{proof}

\begin{remark}\label{rem:unicyclic}
The bound in Theorem~\ref{thm:d-at-most-2} is of the correct order in \(n\) for connected graphs of average degree \(2\), equivalently, for unicyclic graphs. Indeed, let \(S_n^+\) be the graph obtained from the star \(S_n\) by adding an edge between two of its leaves. Then \(S_n^+\) is unicyclic, and a direct computation shows that \(\Spec(L_{S_n^+})=\{0,1^{(n-3)},3,n\}\). Hence \(\lambda_2(S_n^+)/\lambda_n(S_n^+)=1/n\). It would be interesting to determine the optimal upper bound for unicyclic graphs.
\end{remark}

\section{Regular graphs and Ramanujan spectra}\label{sec:regular}
In this section, we prove Theorems~\ref{thm:regular-ratio}--\ref{thm:ramanujan}.
We begin by collecting the necessary tools and notation in the first subsection, and then prove each theorem in the subsequent three subsections.

\subsection{Preparation}\label{sec:preliminaries}
We first introduce the celebrated Courant--Fischer Theorem.
\begin{theorem}[Courant--Fischer Theorem]
    Let $M$ be an $n \times n$ symmetric real matrix with eigenvalues $\lambda_1 \leq \lambda_2 \leq \cdots \leq \lambda_n$.
    Then for every $1 \leq k \leq n$,
    \[
    \lambda_k = \min_{\substack{S \subset \mathbb{R}^n \\ \dim(S) = k}} \max_{\vec{x} \in S} \frac{\vec{x}^\top M \vec{x}}{\vec{x}^\top \vec{x}} = \max_{\substack{T \subset \mathbb{R}^n \\ \dim(T) = n - k + 1}} \min_{\vec{x} \in T} \frac{\vec{x}^\top M \vec{x}}{\vec{x}^\top \vec{x}}.
    \]
\end{theorem}

Let \(G\) be a graph and let \(u,v\) be two of its vertices. We write \(\operatorname{dist}_G(u,v)\) for the \emph{distance} between \(u\) and \(v\), defined as the minimum number of edges in a path of \(G\) connecting them. If no such path exists, we set \(\operatorname{dist}_G(u,v)=\infty\).
To analyze the layered neighbourhood structure used in the upcoming proof, we introduce the following definitions.

\begin{definition}
For a graph $G$, a vertex $u\in V(G)$, and an integer $k\ge 0$, we write
\[
B_G(u,k):=\{v\in V(G):\operatorname{dist}_G(u,v)\le k\}
\]
for the ball of radius $k$ centered at $u$.

For any two edges $e=xy$ and $f=uv$ in $G$, we define their distance by
\[
\operatorname{dist}_G(e,f):=
\min\{\operatorname{dist}_G(x,u),\operatorname{dist}_G(x,v),
      \operatorname{dist}_G(y,u),\operatorname{dist}_G(y,v)\}.
\]
\end{definition}

 In what follows, all balls are ordinary graph balls, and all distances between edges are understood in this sense. For a positive integer \(r\), we write \([r]:=\{1,\dots,r\}\). In the proof of Lemma~\ref{lem:regular-compare}, the set $U_i$ denotes the vertices at distance exactly $i$ from $U_0=\{u_0,u_1\}$; similarly for the sets $V_i$. In particular, \(\bigcup_{i=0}^k U_i \subseteq B_G(u_0,k)\cup B_G(u_1,k)\). For notational convenience, we sometimes identify $B_G(u,k)$ with its vertex set.

\subsection{Proof of Theorem~\ref{thm:regular-ratio}}

We now prove Conjecture~\ref{conj:spielman} for regular graphs. The first step is a comparison lemma that simultaneously controls the Rayleigh quotients relevant to $\lambda_2$ and $\lambda_n$.

\begin{lemma}\label{lem:regular-compare}
Let $d\ge 2$ and $k\ge 1$ be integers. Let $G$ be an $n$-vertex connected $d$-regular graph containing two edges $u_{0}u_{1}$ and $v_{0}v_{1}$ at distance at least $2k + 2$. Then
\[
(\sqrt{d - 1} + 1)^2 \cdot \lambda_{2} \leq (\sqrt{d - 1} - 1)^2 \cdot \lambda_n + \frac{4(d - 1)^{3/2}}{k + 1}.
\]
\end{lemma}

\begin{proof}
Let $U_0 = \{u_0, u_1\}$ and $V_0 = \{v_0, v_1\}$. For $0 \le i \le k$, denote by $U_i$ (resp. $V_i$) the set of vertices at distance exactly $i$ from $U_0$ (resp. $V_0$). The distance condition ensures that the collections $\{U_i\}_{i=0}^k$ and $\{V_j\}_{j=0}^k$ are pairwise disjoint, and $e(U_i, V_j)=0$ for all $0\le i,j\le k$. Let $U_{k+1}$ (resp. $V_{k+1}$) be the set of vertices at distance $k+1$ from $U_0$ (resp. $V_0$); these may intersect, but the edges between $U_k$ and $U_{k+1}$ are disjoint from those between $V_k$ and $V_{k+1}$. In particular, a shortest path joining \(U_0\) to \(V_0\) has length at least \(2k+2\), so \(U_i\neq\varnothing\) and \(V_i\neq\varnothing\) for every \(0\le i\le k+1\).

Recall the variational characterizations
\[
\lambda_2 = \min_{\vec{x}\perp\vec{1}}\frac{\vec{x}^{\top}L_G\vec{x}}{\vec{x}^{\top}\vec{x}},\qquad \text{and} \qquad
\lambda_n = \max_{\vec{y}}\frac{\vec{y}^{\top}L_G\vec{y}}{\vec{y}^{\top}\vec{y}}.
\]

We construct two test vectors $\vec{x}$ (for $\lambda_2$) and $\vec{y}$ (for $\lambda_n$) as follows:
\begingroup
\small
\[
\vec{x}(a)=\begin{cases}
(d-1)^{-i/2} & \text{if } a\in U_i,\;0\le i\le k,\\
-\beta\,(d-1)^{-j/2} & \text{if } a\in V_j,\;0\le j\le k,\\
0 & \text{otherwise},
\end{cases}
\qquad
\vec{y}(a)=\begin{cases}
(-1)^i (d-1)^{-i/2} & \text{if } a\in U_i,\;0\le i\le k,\\
-\beta (-1)^j (d-1)^{-j/2} & \text{if } a\in V_j,\;0\le j\le k,\\
0 & \text{otherwise},
\end{cases}
\]
\endgroup
where $\beta$ is chosen so that $\vec{x}\perp\vec{1}$.

Define $t_i = e(U_i, U_{i+1})/|U_i|$ for $0\le i\le k$ (with $U_{k+1}$ as above), and similarly $s_j = e(V_j, V_{j+1})/|V_j|$ for $0\le j\le k$. Because $G$ is $d$-regular, we have $t_i, s_j \le d-1$ for all $i,j$.

Now compute the Rayleigh quotients. For $\vec{x}$ we have
\[
\vec{x}^{\top}L_G\vec{x} = \sum_{\substack{ab\in E(G)}} (\vec{x}(a)-\vec{x}(b))^2.
\]
Edges entirely inside $U_i$ or $V_j$ contribute $0$ because $\vec{x}$ is constant on each $U_i$ and $V_j$. 
For $0\le i\le k-1$, each edge between $U_i$ and $U_{i+1}$ contributes $\bigl((d-1)^{-i/2}-(d-1)^{-(i+1)/2}\bigr)^2$, while for $i=k$, each edge between $U_k$ and $U_{k+1}$ contributes $(d-1)^{-k}$ (this is because $\vec{x}=0$ on $U_{k+1}$). 
The contributions from the $V$-side are analogous, with an extra factor $\beta^2$. 
No edges connect $U$-layers to $V$-layers. Therefore adding everything up, we obtain
$\vec{x}^{\top}L_G\vec{x} = X_0 + \beta^2 X_1,$
where
\[
\begin{aligned}
X_0 &= \sum_{i=0}^{k-1} t_i |U_i|\bigl((d-1)^{-i/2}-(d-1)^{-(i+1)/2}\bigr)^2 + t_k |U_k| (d-1)^{-k},\\
X_1 &= \sum_{j=0}^{k-1} s_j |V_j|\bigl((d-1)^{-j/2}-(d-1)^{-(j+1)/2}\bigr)^2 + s_k |V_k| (d-1)^{-k}.
\end{aligned}
\]
Similarly, for an edge $ab$ between $U_i$ and $U_{i+1}$, the contribution to $\vec y^{\top}L_G \vec y$ is
$$\bigl((-1)^i(d-1)^{-i/2}-(-1)^{i+1}(d-1)^{-(i+1)/2}\bigr)^2=\bigl((d-1)^{-i/2}+(d-1)^{-(i+1)/2}\bigr)^2.$$
The $V$-side behaves analogously, so
$\vec{y}^{\top}L_G\vec{y} = Y_0 + \beta^2 Y_1$,
with
\[
\begin{aligned}
Y_0 &= \sum_{i=0}^{k-1} t_i |U_i|\bigl((d-1)^{-i/2}+(d-1)^{-(i+1)/2}\bigr)^2 + t_k |U_k| (d-1)^{-k},\\
Y_1 &= \sum_{j=0}^{k-1} s_j |V_j|\bigl((d-1)^{-j/2}+(d-1)^{-(j+1)/2}\bigr)^2 + s_k |V_k| (d-1)^{-k}.
\end{aligned}
\]

Set $D_+ = \sqrt{d-1}+1$, $D_- = \sqrt{d-1}-1$. A direct computation shows
\begin{align*}
X_0 &= \sum_{i=0}^{k-1} \frac{t_i |U_i|}{(d-1)^{i+1}} D_-^2 + \frac{t_k |U_k|}{(d-1)^k}, &
Y_0 &= \sum_{i=0}^{k-1} \frac{t_i |U_i|}{(d-1)^{i+1}} D_+^2 + \frac{t_k |U_k|}{(d-1)^k},\\
X_1 &= \sum_{j=0}^{k-1} \frac{s_j |V_j|}{(d-1)^{j+1}} D_-^2 + \frac{s_k |V_k|}{(d-1)^k}, &
Y_1 &= \sum_{j=0}^{k-1} \frac{s_j |V_j|}{(d-1)^{j+1}} D_+^2 + \frac{s_k |V_k|}{(d-1)^k}.
\end{align*}
Moreover, we have
$\vec{x}^{\top}\vec{x} = Z_0 + \beta^2 Z_1$ and
$\vec{y}^{\top}\vec{y} = Z_0 + \beta^2 Z_1$,
where 
\[
Z_0 = \sum_{i=0}^k |U_i|(d-1)^{-i} \qquad \text{and} \qquad  Z_1 = \sum_{j=0}^k |V_j|(d-1)^{-j}.
\]
Since $\vec{x}\perp\vec{1}$,
by the Courant--Fischer theorem, we have
\[
\lambda_2 \le \varphi(\vec{x}) \coloneqq \frac{X_0 + \beta^2 X_1}{Z_0 + \beta^2 Z_1}, \qquad \text{and} \qquad
\lambda_n \ge \varphi(\vec{y}) \coloneqq \frac{Y_0 + \beta^2 Y_1}{Z_0 + \beta^2 Z_1}.
\]

The following two claims establish key inequalities among these parameters defined above.

\begin{claim}\label{cl:U}
\(
D_+^2 X_0 - D_-^2 Y_0 \le \frac{4(d-1)^{3/2}}{k+1}\,Z_0.
\)
\end{claim}
\begin{proof}
From the expressions above, all terms with \(0 \le i \le k-1\) cancel with the symmetric expression \(D_+^2 X_0 - D_-^2 Y_0\).
Only the boundary terms remain:
\[
D_+^2 X_0 - D_-^2 Y_0 = D_+^2\cdot\frac{t_k|U_k|}{(d-1)^k} - D_-^2\cdot\frac{t_k|U_k|}{(d-1)^k}
= \frac{t_k|U_k|}{(d-1)^k}\,(D_+^2 - D_-^2)= \frac{t_k|U_k|}{(d-1)^k}\,4\sqrt{d-1}.
\]
Since $G$ is $d$-regular, we have $t_k \le d-1$. Moreover, every vertex of $U_{i+1}$ has at least one neighbour in $U_i$, so $|U_{i+1}| \le e(U_i,U_{i+1}).$
On the other hand, each vertex of $U_i$ has at most $d-1$ neighbours in $U_{i+1}$: for $i\ge 1$ this is because it already has a neighbour in $U_{i-1}$, while for $i=0$ each of $u_0,u_1$ has exactly one neighbour inside $U_0$. Hence
$e(U_i,U_{i+1}) \le (d-1)|U_i|
$ for $0\le i\le k-1$.

Therefore
$|U_{i+1}| \le (d-1)|U_i|
$ for $0\le i\le k-1,$
so the sequence $\frac{|U_i|}{(d-1)^i}$ is non-increasing. Therefore its minimum at $i=k$ is at most the average:
\[
\frac{|U_k|}{(d-1)^k} \le \frac{1}{k+1}\sum_{i=0}^k \frac{|U_i|}{(d-1)^i} = \frac{Z_0}{k+1}.
\]
Consequently,
$D_+^2 X_0 - D_-^2 Y_0 \le 4\sqrt{d-1}\,(d-1)\,\frac{Z_0}{k+1}
= \frac{4(d-1)^{3/2}}{k+1}\,Z_0,$
which proves the claim.
\end{proof}

The same argument applied to the $V$-side yields the following analogous bound, where the proof is omitted.
\begin{claim}\label{cl:V}
\(
D_+^2 X_1 - D_-^2 Y_1 \le \frac{4(d-1)^{3/2}}{k+1}\,Z_1.
\)
\end{claim}

Multiplying Claim~\ref{cl:U} by $1$ and Claim~\ref{cl:V} by $\beta^2$, then adding, we obtain
\[
D_+^2 (X_0+\beta^2 X_1) - D_-^2 (Y_0+\beta^2 Y_1)
\le \frac{4(d-1)^{3/2}}{k+1}\,(Z_0+\beta^2 Z_1).
\]
Further dividing by $Z_0+\beta^2 Z_1$ gives
$$D_+^2\,\varphi(\vec{x}) - D_-^2\,\varphi(\vec{y}) \le \frac{4(d-1)^{3/2}}{k+1}, \quad \text{or equivalently} \quad D_-^2\,\varphi(\vec{y}) + \frac{4(d-1)^{3/2}}{k+1} \ge D_+^2\,\varphi(\vec{x}).$$
Finally, using $\lambda_2 \le \varphi(\vec{x})$ and $\lambda_n \ge \varphi(\vec{y})$,  we conclude that
\[
D_-^2\,\lambda_n + \frac{4(d-1)^{3/2}}{k+1}
\ge D_-^2\,\varphi(\vec{y}) + \frac{4(d-1)^{3/2}}{k+1}
\ge D_+^2\,\varphi(\vec{x})
\ge D_+^2\,\lambda_2,
\]
which is exactly the desired inequality for Lemma~\ref{lem:regular-compare}. 
\end{proof}

Now we can rapidly prove Theorem~\ref{thm:regular-ratio}.

\begin{proof}[Proof of Theorem~\ref{thm:regular-ratio}]
If \(G\) is disconnected, then \(\lambda_2(G)=0\), and the desired inequality is trivial. Hence we may assume that \(G\) is connected. Let
    $D_+=\sqrt{d-1}+1$, and $D_-=\sqrt{d-1}-1$.
    By Lemma~\ref{lem:regular-compare}, we have 
    \[
    D_+^2\,\lambda_2 \le D_-^2\,\lambda_n + \frac{4(d-1)^{3/2}}{k+1}.
    \]
    Dividing both sides by $\lambda_n D_+^2$ and using the inequalities $\lambda_n \ge d > d-1$ and $D_+^2 > d-1$, we obtain
    \[
    \frac{\lambda_2}{\lambda_n}
    \le \frac{D_-^2}{D_+^2} + \frac{4(d-1)^{3/2}}{(k+1)\,\lambda_n D_+^2}
    < \frac{(\sqrt{d-1}-1)^2}{(\sqrt{d-1}+1)^2} + \frac{4}{(k+1)\sqrt{d-1}},
    \]
    as required.
\end{proof}

\subsection{Proof of Theorem~\ref{thm:higher-eigs}}
We prove parts \textup{(a)} and \textup{(b)} separately.
\smallskip

    \noindent\textbf{(a)} Since $d\ge 3$, we have $D_+ = \sqrt{d-1}+1 > 0$ and $D_- = \sqrt{d-1}-1 > 0$. Let \(G\) be an \(n\)-vertex \(d\)-regular graph containing \(4s-2\) edges that are pairwise at distance at least \(2k+2\). For \(1 \le i \le 2s-1\), let \(e_i = u_{i0}u_{i1}\) and \(f_i = v_{i0}v_{i1}\) be these edges.
    Define the vertex sets
    \[
    Q_i = B_G(u_{i0},k) \cup B_G(u_{i1},k),\qquad
S_i = B_G(v_{i0},k) \cup B_G(v_{i1},k) \qquad \text{for} \quad 1\le i\le 2s-1,
    \]
    which are pairwise disjoint by construction. 
    If $x\in Q_i\cup S_i$ and $y\in Q_{i'}\cup S_{i'}$ with $i\neq i'$, then
    $\operatorname{dist}_G(x,y)\ge 2k+2-k-k=2,$
    so there is no edge between $Q_i\cup S_i$ and $Q_{i'}\cup S_{i'}$.
    For each $1\leq i\leq 2s-1$, let $U_{i0}=\{u_{i0},u_{i1}\}$ and $V_{i0}=\{v_{i0},v_{i1}\}$.
    Denote by $U_{ir}$ (resp. $V_{ij}$) the set of vertices at distance exactly $r$ from $U_{i0}$ (resp. $j$ from $V_{i0}$).
Note that $U_{ir} \subseteq Q_i$, $V_{ij} \subseteq S_i$ are disjoint.
    Then the Courant--Fischer theorem gives, for every $s\ge 1$,
    $$\lambda_{s+1}=\min_{\dim(S)=s+1}\max_{\vec{x}\in S}\frac{\vec{x}^{\top}L_G\vec{x}}{\vec{x}^{\top}\vec{x}}=\min_{\substack{\dim(S)=s\\ S\perp\vec{1}}}\max_{\vec{x}\in S}\frac{\vec{x}^{\top}L_G\vec{x}}{\vec{x}^{\top}\vec{x}}, \qquad \lambda_{n-s+1}=\max_{\dim(T)=s}\min_{\vec{y}\in T}\frac{\vec{y}^{\top}L_G\vec{y}}{\vec{y}^{\top}\vec{y}},$$
    implying that
   \[
\lambda_{s+1}\le\max_{\vec{x}\in S}\frac{\vec{x}^{\top}L_G\vec{x}}{\vec{x}^{\top}\vec{x}}\quad(\forall S\perp\vec{1},\ \dim(S)=s),\quad \text{and} \quad 
\lambda_{n-s+1}\ge\min_{\vec{y}\in T}\frac{\vec{y}^{\top}L_G\vec{y}}{\vec{y}^{\top}\vec{y}}\quad(\forall T,\ \dim(T)=s).
\]
    For each $t$ we define test vectors $\vec{x}_t$ (for $\lambda_{s+1}$) and $\vec{y}_t$ (for $\lambda_{n-s+1}$) as follows:
    \begingroup
\small
\[
\vec{x}_t(a)=\begin{cases}
(d-1)^{-i/2} & \text{if } a\in U_{ti},\; 0\le i\le k,\\
-\beta_t\,(d-1)^{-j/2} & \text{if } a\in V_{tj},\; 0\le j\le k,\\
0 & \text{otherwise},
\end{cases}
\qquad
\vec{y}_t(a)=\begin{cases}
(-1)^i (d-1)^{-i/2} & \text{if } a\in U_{ti},\; 0\le i\le k,\\
-\beta_t (-1)^j (d-1)^{-j/2} & \text{if } a\in V_{tj},\; 0\le j\le k,\\
0 & \text{otherwise},
\end{cases}
\]
\endgroup
where \(\beta_t\) is chosen so that \(\vec{x}_t\perp\vec{1}\). For each \(t\), the vectors \(\vec{x}_t\) and \(\vec{y}_t\) have the same support. For distinct indices \(t\neq t'\), these supports are disjoint, and moreover there is no edge between the support corresponding to \(t\) and the support corresponding to \(t'\).

    Repeating the computation of Lemma~\ref{lem:regular-compare} yields, for every $t$,
 \begin{equation}\label{eq:pair-bound}
        D_-^2\,\varphi(\vec{y}_t)+\frac{4(d-1)^{3/2}}{k+1}\ge D_+^2\,\varphi(\vec{x}_t),
        \qquad\text{with } \quad\varphi(\vec{z}) = \frac{\vec{z}^{\top}L_G\vec{z}}{\vec{z}^{\top}\vec{z}} .
\end{equation}

    Let $L,M\subseteq[2s-1]$ be index sets of size $s-1$ such that
    \[
    \varphi(\vec{x}_\ell)\le\varphi(\vec{x}_{\ell'})\text{ for all }\ell\in L,\;\ell'\notin L,\qquad
    \varphi(\vec{y}_m)\ge\varphi(\vec{y}_{m'})\text{ for all }m\in M,\;m'\notin M .
    \]
    Since $2s-1 > 2(s-1)$, the set $N=[2s-1]\setminus(L\cup M)$ is nonempty. Pick $w\in N$ and define
    \[
     S=\operatorname{span}\bigl(\{\vec{x}_w\}\cup\{\vec{x}_\ell:\ell\in L\}\bigr),\qquad
    T=\operatorname{span}\bigl(\{\vec{y}_w\}\cup\{\vec{y}_m:m\in M\}\bigr).
    \]
   Since there are no edges between distinct supports, every nonzero vector
   $ \vec{x}=\beta\vec{x}_w+\sum_{\ell\in L}\alpha_\ell\vec{x}_\ell\in S$
    satisfies
    \[
    \varphi(\vec{x})=
    \frac{
    \beta^2\|\vec{x}_w\|^2\varphi(\vec{x}_w)
    +
    \sum_{\ell\in L}\alpha_\ell^2\|\vec{x}_\ell\|^2\varphi(\vec{x}_\ell)
    }{
    \beta^2\|\vec{x}_w\|^2+\sum_{\ell\in L}\alpha_\ell^2\|\vec{x}_\ell\|^2}
    \le \varphi(\vec{x}_w).
    \]
    Since $w\notin L$, we have $\varphi(\vec{x}_\ell)\le \varphi(\vec{x}_w)$ for any $\ell\in L$, and therefore $\max_{\vec{x}\in S}\varphi(\vec{x}) = \varphi(\vec{x}_w).$

    Similarly, every nonzero vector $\vec{y}=\gamma\vec{y}_w+\sum_{m\in M}\delta_m\vec{y}_m\in T$ satisfies
    \[
    \varphi(\vec{y})=
    \frac{
    \gamma^2\|\vec{y}_w\|^2\varphi(\vec{y}_w)
    +
    \sum_{m\in M}\delta_m^2\|\vec{y}_m\|^2\varphi(\vec{y}_m)
    }{
    \gamma^2\|\vec{y}_w\|^2+\sum_{m\in M}\delta_m^2\|\vec{y}_m\|^2
    }\ge \varphi(\vec{y}_w) .
    \]
    Since $w\notin M$, we have $\varphi(\vec{y}_w)\le \varphi(\vec{y}_m)$ for any $m\in M$, and therefore $\min_{\vec{y}\in T}\varphi(\vec{y}) = \varphi(\vec{y}_w)$.

    Now $S\perp\vec{1}$ and $\dim(S)=s$, while $\dim(T)=s$. Applying the Courant--Fischer characterizations together with~\eqref{eq:pair-bound} gives
\[
D_+^2\lambda_{s+1} - D_-^2\lambda_{n-s+1}
\le D_+^2\max_{\vec{x}\in S}\varphi(\vec{x}) - D_-^2\min_{\vec{y}\in T}\varphi(\vec{y})
= D_+^2\varphi(\vec{x}_w) - D_-^2\varphi(\vec{y}_w)
\le \frac{4(d-1)^{3/2}}{k+1}.
\]
This is exactly statement (a).

    \medskip

\noindent\textbf{(b)} Throughout, set
\[
M_+=\sqrt{d-1}+2,\qquad M_-=\sqrt{d-1}-2,\qquad \alpha=\frac{\pi}{2k}.
\]
Let $v_1,\dots,v_{2s+1}$ be vertices of $G$ whose pairwise distances are at least $4k$.
For each $\ell\in\{1,\dots,2s+1\}$ and $0\le i\le 2k$, let $V_{\ell i}$ be the set of vertices at distance exactly $i$ from $v_\ell$, and define
\[
x_i:=\frac{\sin(i\alpha)}{(d-1)^{i/2}}\qquad(0\le i\le 2k).
\]
We claim $x_0=x_{2k}=0$, and $x_1\ge x_2\ge \cdots \ge x_{2k-1}>0$. We have 
\[
x_{i-1}+(d-1)x_{i+1}=2\sqrt{d-1}\cos\alpha\,x_i
\qquad(1\le i\le 2k-1).
\]
Also, by the identity \(\sin((i+1)\alpha)=2\cos\alpha\cdot\sin(i\alpha)-\sin((i-1)\alpha)\), and since \(\sin((i-1)\alpha)\ge 0\) for \(1\le i\le 2k-1\), we obtain \(\frac{\sin((i+1)\alpha)}{\sin(i\alpha)}\le 2\cos\alpha\). Hence,
\[
\frac{x_{i+1}}{x_i}
=\frac{\sin((i+1)\alpha)}{\sqrt{d-1}\sin(i\alpha)}
\le \frac{2\cos\alpha}{\sqrt{d-1}}
\le \frac{2}{\sqrt{d-1}}
\qquad(1\le i\le 2k-1).
\]
Since \(d\ge 5\), this proves \(x_1\ge x_2\ge\cdots\ge x_{2k-1}>0\).
Define vectors $\vec{x}_\ell,\vec{y}_\ell\in\mathbb{R}^{V(G)}$ for $\ell \in \{ 1,\dots, 2s+1\}$ by
\[
\vec{x}_\ell(v)=
\begin{cases}
x_i, & v\in V_{\ell i}\ \text{for some }0\le i\le 2k,\\
0,   & \text{otherwise},
\end{cases}
\qquad
\vec{y}_\ell(v)=
\begin{cases}
(-1)^i x_i, & v\in V_{\ell i}\ \text{for some }0\le i\le 2k,\\
0,          & \text{otherwise}.
\end{cases}
\]
Since $x_0=x_{2k}=0$, the supports of both $\vec{x}_\ell$ and $\vec{y}_\ell$ are contained in
$\bigcup_{i=1}^{2k-1}V_{\ell i}$.
If $u\in V_{\ell i}$ and $v\in V_{m j}$ with $\ell\ne m$ and $1\le i,j\le 2k-1$, then
\[
\operatorname{dist}_G(u,v)\ge \operatorname{dist}_G(v_\ell,v_m)-i-j
\ge 4k-(2k-1)-(2k-1)=2.
\]
Hence the supports corresponding to different indices are pairwise disjoint, and there are no edges between distinct supports. In particular,
$\|\vec{x}_\ell\|^2=\|\vec{y}_\ell\|^2=\sum_{i=1}^{2k-1}|V_{\ell i}|\,x_i^2$.
For $1\le i\le 2k-1$, write
\[
n_{\ell i}:=|V_{\ell i}|,\qquad
a_{\ell i}:=\widehat{e}(V_{\ell(i-1)},V_{\ell i}),\qquad
b_{\ell i}:=\widehat{e}(V_{\ell i},V_{\ell i}),\qquad
c_{\ell i}:=\widehat{e}(V_{\ell i},V_{\ell(i+1)}),
\]
where $\widehat{e}(X,Y)$ denotes the number of ordered pairs $(u,v)$ with $u\in X$, $v\in Y$, and $uv\in E(G)$.
Then, we have
\[
a_{\ell i}+b_{\ell i}+c_{\ell i}=d\,n_{\ell i},
\quad \text{and} \quad
a_{\ell i}\ge n_{\ell i},
\]
because $G$ is $d$-regular and every vertex of $V_{\ell i}$ has at least one neighbour in $V_{\ell(i-1)}$.

We claim that for every $\ell \in \{1, \dots ,2s+1 \}$,
\begin{equation}\label{eq:claim-thm2b}
M_+\vec{x}_\ell^{\top}A_G\vec{x}_\ell
-
M_-\vec{y}_\ell^{\top}A_G\vec{y}_\ell
\ge
4(d-1)\cos\alpha\,\|\vec{x}_\ell\|^2.
\end{equation}
Indeed, a direct calculation gives
\begin{equation}\label{eq:adj-bound-local}
M_+\vec{x}_\ell^{\top}A_G\vec{x}_\ell - M_-\vec{y}_\ell^{\top}A_G\vec{y}_\ell = \sum_{i=1}^{2k-1} \bigl(2\sqrt{d-1}\,a_{\ell i}x_{i-1} + 4\,b_{\ell i}x_i + 2\sqrt{d-1}\,c_{\ell i}x_{i+1}\bigr)x_i.
\end{equation}
For $2\le i\le 2k-2$, the ratio bound above gives
$\frac{x_i}{x_{i-1}}\le \frac{2}{\sqrt{d-1}}$ and $\frac{x_i}{x_{i+1}} \ge \frac{\sqrt{d-1}}{2}$,
and hence
\[
2\sqrt{d-1}\,x_{i-1}x_i \ge \frac{8}{\sqrt{d-1}}\,x_{i-1}x_i \ge 4x_i^2 \ge 2\sqrt{d-1}\,x_i x_{i+1},
\]
where the first inequality uses $d\ge5$. Therefore, for $2\le i\le 2k-2$, under the constraints
$a_{\ell i}\ge n_{\ell i}$ and $a_{\ell i}+b_{\ell i}+c_{\ell i}=d\,n_{\ell i}$,
the $i$th summand is minimized when
\[
a_{\ell i}=n_{\ell i},\qquad b_{\ell i}=0,\quad \text{and} \quad c_{\ell i}=(d-1)n_{\ell i}.
\]

If \(k\ge 2\), then for \(i=2k-1\) we have \(x_{2k}=0\), and the ratio bound for \(i=2k-2\) gives
\[
2\sqrt{d-1}\,x_{2k-2}x_{2k-1}\ge 4x_{2k-1}^2 \ge 0
=2\sqrt{d-1}\,x_{2k-1}x_{2k}.
\]
Hence the same minimizing choice
\[
a_{\ell,2k-1}=n_{\ell,2k-1},\qquad b_{\ell,2k-1}=0,\qquad c_{\ell,2k-1}=(d-1)n_{\ell,2k-1}
\]
is again valid.

For $i=1$, we use the extra fact that $V_{\ell0}=\{v_\ell\}$, so every vertex of $V_{\ell1}$ has exactly one neighbour in $V_{\ell0}$; hence $a_{\ell1}=n_{\ell1}$. Thus $b_{\ell1}+c_{\ell1}=(d-1)n_{\ell1}$. Since $2\sqrt{d-1}\,x_1x_2\le 4x_1^2$, the $i=1$ summand is minimized when $b_{\ell1}=0$, $c_{\ell1}=(d-1)n_{\ell1}$.

Combining these cases (when \(k=1\), only the \(i=1\) case is present), by \eqref{eq:adj-bound-local} we obtain
\[
\begin{aligned}
M_+\vec{x}_\ell^{\top}A_G\vec{x}_\ell - M_-\vec{y}_\ell^{\top}A_G\vec{y}_\ell
&\ge \sum_{i=1}^{2k-1} 2\sqrt{d-1}\,n_{\ell i}\bigl(x_{i-1}+(d-1)x_{i+1}\bigr)x_i \\
&= 4(d-1)\cos\alpha\sum_{i=1}^{2k-1}n_{\ell i}x_i^2 = 4(d-1)\cos\alpha\,\|\vec{x}_\ell\|^2,
\end{aligned}
\]
which proves \eqref{eq:claim-thm2b}.
Now for $z\neq 0$, set
\[
\psi(\vec{z}):=\frac{\vec{z}^{\top}A_G\vec{z}}{\vec{z}^{\top}\vec{z}}.
\]
Choose $X\subseteq\{1,\dots,2s+1\}$ with $|X|=s$ so that $\{\psi(\vec{x}_\ell):\ell\in X\}$ consists of the $s$ largest values among $\psi(\vec{x}_1),\dots,\psi(\vec{x}_{2s+1})$.
Choose $Y\subseteq\{1,\dots,2s+1\}$ with $|Y|=s-1$ so that $\{\psi(\vec{y}_\ell):\ell\in Y\}$ consists of the $s-1$ smallest values among $\psi(\vec{y}_1),\dots,\psi(\vec{y}_{2s+1})$.
Since $|X|+|Y|=2s-1<2s+1$, there exists $w\in\{1,\dots,2s+1\}\setminus(X\cup Y)$.
Define
\[
U=\operatorname{span}\bigl(\{\vec{x}_w\}\cup\{\vec{x}_\ell:\ell\in X\}\bigr),
\qquad
V=\operatorname{span}\bigl(\{\vec{y}_w\}\cup\{\vec{y}_\ell:\ell\in Y\}\bigr).
\]
Because the corresponding supports of $\vec{x}_\ell$ (and of $\vec{y}_\ell $) are pairwise disjoint, and there is no edge between distinct supports, every nonzero vector
$u=\beta\vec{x}_w+\sum_{\ell\in X}\alpha_\ell\vec{x}_\ell\in U$
satisfies
\[
\psi(u)=
\frac{
\beta^2\|\vec{x}_w\|^2\psi(\vec{x}_w)
+
\sum_{\ell\in X}\alpha_\ell^2\|\vec{x}_\ell\|^2\psi(\vec{x}_\ell)
}{
\beta^2\|\vec{x}_w\|^2+\sum_{\ell\in X}\alpha_\ell^2\|\vec{x}_\ell\|^2
},
\]
Since $w\notin X$, we have $\psi(\vec{x}_w)\le \psi(\vec{x}_\ell)$ for every $\ell\in X$. Therefore $\min_{u\in U\setminus\{0\}}\psi(u)=\psi(\vec{x}_w)$. Similarly, since $w\notin Y$, we have $\psi(\vec{y}_w)\ge \psi(\vec{y}_\ell)$ for all $\ell\in Y$; hence $\max_{v\in V\setminus\{0\}}\psi(v)=\psi(\vec{y}_w)$.

Because the adjacency eigenvalues are ordered non-increasingly, Courant--Fischer takes the form
\[
\mu_{s+1}=\max_{\dim U=s+1}\min_{u\in U\setminus\{0\}}\psi(u),\qquad
\mu_{n-s+1}=\min_{\dim V=s}\max_{v\in V\setminus\{0\}}\psi(v),
\]
equivalently by applying the usual theorem to $-A_G$. Since $\dim U=s+1$ and $\dim V=s$, this gives
\[
\mu_{s+1}\ge \min_{u\in U\setminus\{0\}}\psi(u)=\psi(\vec{x}_w),
\qquad
\mu_{n-s+1}\le \max_{v\in V\setminus\{0\}}\psi(v)=\psi(\vec{y}_w).
\]
Applying \eqref{eq:claim-thm2b} with $\ell=w$, we obtain
\[
M_+\mu_{s+1}-M_-\mu_{n-s+1}
\ge
M_+\psi(\vec{x}_w)-M_-\psi(\vec{y}_w)
\ge
4(d-1)\cos\alpha.
\]
This completes the proof of Theorem~\ref{thm:higher-eigs}.
\qed

\subsection{Proof of Theorem~\ref{thm:ramanujan}}
It suffices to consider $0<\varepsilon\le 1$.
Throughout this proof, fix $d\ge 3$ and set
\[
\eta:=\frac{\varepsilon}{2}.
\]
We begin by proving the theorem for all \(d\ge 3\) via
Theorem~\ref{thm:higher-eigs}\textup{(a)}. After that, assuming \(d\ge 6\), we use Theorem~\ref{thm:higher-eigs}\textup{(b)} to obtain a sharper quantitative lower bound on the constant \(\beta\).

\bigskip

\noindent\underline{\textbf{Negative eigenvalues for $d\ge 3$.}}
Define
\(M_d:=\frac{4(d-1)^{3/2}}{d-2\sqrt{d-1}}\)
and choose
$k:=\left\lceil M_d\,\eta^{-1}\right\rceil $.
Then
\begin{equation}\label{eq:eta-md-k}
\frac{4(d-1)^{3/2}}{(k+1)(d-2\sqrt{d-1})} \le \frac{M_d}{k} \le \eta.
\end{equation}
Let $\mathcal E$ be a maximal set of edges with pairwise distance at least $2k+2$.
For $e=uv\in\mathcal E$, let $T(e)$ be the set of edges with distance at most $2k+2$ from $e$.
Each edge in $T(e)$ has an endpoint in $B_G(u,2k+2)\cup B_G(v,2k+2)$, so we have
\[
|T(e)|
\le
2d\left(1+d\sum_{i=0}^{2k+1}(d-1)^i\right)
\le\frac{2d^2}{d-2}\cdot (d-1)^{2k+2}
\]
Because $\mathcal E$ is maximal, the family $\{T(e):e\in\mathcal E\}$ covers $E(G)$.
Therefore
\[
\frac{nd}{2}=|E(G)|\le |\mathcal E|\,\frac{2d^2}{d-2} (d-1)^{2k+2},
\]
and hence
\[
|\mathcal E|
\ge
\frac{n(d-2)}{4d}(d-1)^{-(2k+2)}.
\]

Set
$s:=\left\lfloor \frac{|\mathcal E|+2}{4}\right\rfloor$.
Then $4s-2\le |\mathcal E|$, so we may choose $4s-2$ edges from $\mathcal E$
and apply Theorem~\ref{thm:higher-eigs}(a) to them. Also,
\[
s\ge \frac{|\mathcal E|}{4}-\frac12
\ge \frac{n(d-2)}{16d}(d-1)^{-(2k+2)}-\frac12.
\]
If
\[
n\ge n_1(d,\eta):=\left\lceil \frac{16d}{d-2}(d-1)^{2k+2}\right\rceil,
\]
then
\[
\frac{n(d-2)}{16d}(d-1)^{-(2k+2)}\ge 1, \qquad \text{and therefore} \qquad s\ge \frac{n(d-2)}{32d}(d-1)^{-(2k+2)}.
\]

Since $G$ is Ramanujan and $s\ge 1$, we have
\(\mu_{s+1}\le \mu_2\le 2\sqrt{d-1}\),
hence
\(\lambda_{s+1}=d-\mu_{s+1}\ge d-2\sqrt{d-1}.\)
Theorem~\ref{thm:higher-eigs}(a) gives
\[
(d-2\sqrt{d-1})\lambda_{n-s+1}
\ge(d+2\sqrt{d-1})\lambda_{s+1}-\frac{4(d-1)^{3/2}}{k+1}
\ge (d+2\sqrt{d-1})(d-2\sqrt{d-1})-\frac{4(d-1)^{3/2}}{k+1}.
\]
It follows from~\eqref{eq:eta-md-k} that
\[
\lambda_{n-s+1}
\ge
d+2\sqrt{d-1}-\frac{4(d-1)^{3/2}}{(k+1)(d-2\sqrt{d-1})}
\ge d+2\sqrt{d-1}-\eta.
\]
Therefore
\[
\mu_{n-s+1}=d-\lambda_{n-s+1}\le -2\sqrt{d-1}+\eta
< -2\sqrt{d-1}+\varepsilon.
\]
So $G$ has at least $s$ adjacency eigenvalues smaller than
$-2\sqrt{d-1}+\varepsilon$.
Since $0<\varepsilon\le 1$, we have $\eta^{-1}\le 2\varepsilon^{-1}$ and $\eta^{-1}\ge 1$.
Hence
\[
k=\left\lceil M_d\,\eta^{-1}\right\rceil
\le M_d\,\eta^{-1}+1
\le (2M_d+1)\varepsilon^{-1}.
\]
Define
\[
c_1(d,\eta):=\frac{d-2}{32d}(d-1)^{-(2k+2)}.
\]
Then every $n$-vertex $d$-regular Ramanujan graph with $n\ge n_1(d,\eta)$ has at least
$c_1(d,\eta)\cdot n$ adjacency eigenvalues smaller than
$-2\sqrt{d-1}+\varepsilon$.
After absorbing constant factors depending only on $d$, we obtain
\[
c_1(d,\eta)\ge {c_d}^{-\frac{1}{\varepsilon}}
\]
for some constant $c_d>1$ depending only on $d$.

\bigskip

\noindent\underline{\textbf{Improved bound for negative eigenvalues when $d\ge 6$.}}
Now assume $d\ge 6$ and define
\[
N_d:=\pi\sqrt{\frac{d-1}{2(\sqrt{d-1}-2)}}.
\]
Choose $k = \left\lceil \frac{N_d}{\sqrt{\eta}}\right\rceil$.
Using again $1-\cos t\le t^2/2$, we get
\begin{equation}\label{eq:eta-choice}
\frac{4(d-1)}{\sqrt{d-1}-2}\Bigl(1-\cos\frac{\pi}{2k}\Bigr)
\le
\frac{4(d-1)}{\sqrt{d-1}-2}\cdot \frac12\Bigl(\frac{\pi}{2k}\Bigr)^2
=\frac{N_d^2}{k^2}\le \eta.
\end{equation}

Let $\mathcal V$ be a maximal set of vertices whose pairwise distances are at least $4k$.
Because $\mathcal V$ is maximal, the balls $B_G(v,4k)$ for all $v\in\mathcal V$, cover $V(G)$.
Also,
\[
|B_G(v,4k)|
\le
1+d\sum_{i=0}^{4k-1}(d-1)^i
\le\frac{d}{d-2}\cdot (d-1)^{4k}.
\]
Hence
\[
|\mathcal V|
\ge
\frac{n(d-2)}{d(d-1)^{4k}}.
\]

Set
$s:=\left\lfloor \frac{|\mathcal V|-1}{2}\right\rfloor$ .
Then $2s+1\le |\mathcal V|$, so we may choose $2s+1$ vertices from $\mathcal V$
and apply Theorem~\ref{thm:higher-eigs}(b) to them. Also,
\[
s\ge \frac{|\mathcal V|-2}{2}
\ge \frac{n(d-2)}{2d}(d-1)^{-4k}-1.
\]
If
\[
n\ge n_2(d,\eta):=\left\lceil \frac{4d}{d-2}(d-1)^{4k}\right\rceil,
\]
then
\[
\frac{n(d-2)}{2d}(d-1)^{-4k}\ge 2,
\quad \text{and hence} \quad
s\ge \frac{n(d-2)}{4d}(d-1)^{-4k}.
\]
Since $G$ is Ramanujan and $s\ge 1$, $\mu_{s+1}\le \mu_2\le 2\sqrt{d-1}$.
Applying Theorem~\ref{thm:higher-eigs}(b), and by \eqref{eq:eta-choice} we obtain
\[
\begin{aligned}
\mu_{n-s+1}
&\le
\frac{(\sqrt{d-1}+2)\mu_{s+1}-4(d-1)\cos(\pi/(2k))}{\sqrt{d-1}-2} \\
&\le
-2\sqrt{d-1}
+
\frac{4(d-1)}{\sqrt{d-1}-2}\Bigl(1-\cos\frac{\pi}{2k}\Bigr)
\le
-2\sqrt{d-1}+\eta
<
-2\sqrt{d-1}+\varepsilon.
\end{aligned}
\]
Thus $G$ has at least $s$ adjacency eigenvalues smaller than
$-2\sqrt{d-1}+\varepsilon$.
Since $0<\varepsilon\le 1$, we have $\eta^{-1/2}=\sqrt{2}\,\varepsilon^{-1/2}\ge 1$.
Therefore
\[
k=\left\lceil N_d\,\eta^{-1/2}\right\rceil
\le N_d\,\eta^{-1/2}+1
\le (\sqrt{2}N_d+1)\varepsilon^{-1/2}.
\]
Define
\[
c_2(d,\eta):=\frac{d-2}{4d}(d-1)^{-4k}.
\]
Then every $n$-vertex $d$-regular Ramanujan graph with $n\ge n_2(d,\eta)$ has at least $c_2(d,\eta)\cdot n$ adjacency eigenvalues smaller than $-2\sqrt{d-1}+\varepsilon$. After adjusting constants depending only on $d$, we obtain
\[
c_2(d,\eta)\ge {c_d}^{-\frac{1}{\sqrt{\varepsilon}}}
\]
for some constant $c_d>1$ depending only on $d$. This completes the proof of Theorem~\ref{thm:ramanujan}.
\qed

\section{Trees and the You--Liu conjecture}\label{sec:trees}

In this section, we prove Theorem~\ref{thm:tree-ratio}, thereby confirming Conjecture~\ref{conj:you-liu}.

\begin{proof}[Proof of Theorem~\ref{thm:tree-ratio}]
For \(n=3\), the only tree is \(P_3=S_3\), so the statement is immediate.
Henceforth we may assume \(n\ge 4\).

For the $n$-vertex star \(S_n\), we have \(\operatorname{Spec}(L_{S_n})=\{0,1^{(n-2)},n\}\), implying that \(\lambda_2(S_n)/\lambda_n(S_n)=1/n\). 
Thus it remains to prove that every \(n\)-vertex non-star tree \(T\) satisfies
$\frac{\lambda_2(T)}{\lambda_n(T)}<\frac1n$.

\medskip
\noindent
\underline{\textbf{Choosing a centroid vertex.}}
Recall that a vertex \(v\in V(T)\) is called a \emph{centroid} of \(T\) if every connected component of \(T-v\) has size at most \(|V(T)|/2\). Let \(v\) be a centroid of \(T\), and write $d:=\deg(v)$. After deleting \(v\), let the connected components of \(T-v\) be \(C_1,\dots,C_d\), and write \(s_i:=|C_i|\) for \(1\le i\le d\), where, after relabeling if necessary, \(s_1\ge s_2\ge \cdots \ge s_d\ge 1\). Set
\[s:=s_1 \qquad \text{and} \qquad t:=n-s-1.\]
Since \(v\) is a centroid, every component of \(T-v\) has size at most \(n/2\), hence $s\le \frac n2$. Also, since \(s\) is the largest among the \(d\) positive integers \(s_1,\dots,s_d\) with sum \(n-1\), we have $s\ge \frac{n-1}{d}$. Because \(T\) is not a star, we have \(d\le n-2\). Moreover, \(v\) cannot be a leaf: indeed, if \(\deg(v)=1\), then \(T-v\) has a component of size \(n-1>n/2\), contradicting the centroid property. Therefore \(d\ge 2\). Altogether, we have
\begin{equation}\label{eq:s-range}
\frac{n-1}{d}\le s\le \frac n2\qquad \text{and} \qquad
2\le d\le n-2.
\end{equation}

\medskip
\noindent
\underline{\textbf{An upper bound for \(\lambda_2(T)\).}}
Let $B:=\bigcup_{i=2}^d C_i$, so that \(|B|=t=n-s-1\). Consider vectors \(\vec{x}\in \mathbb{R}^{V(T)}\) that are constant on each of the three blocks \(C_1\), \(\{v\}\), and \(B\), say
\[
x_u=
\begin{cases}
a,& u\in C_1,\\
c,& u=v,\\
b,& u\in B.
\end{cases}
\]
If in addition \(\vec{x}\perp \vec{1}\), then
\begin{equation}\label{eq:orthogonality}
sa+c+tb=0.
\end{equation}

Since \(\vec{x}\) is constant on \(C_1\) and on \(B\), only edges joining different
blocks contribute to \(\vec{x}^{\top}L_{T}\vec{x}\). There is exactly one edge between
\(v\) and \(C_1\), and exactly \(d-1\) edges between \(v\) and \(B\). Hence
\begin{equation}\label{eq:energy}
\vec{x}^{\top}L_{T}\vec{x}=(a-c)^2+(d-1)(b-c)^2,
\end{equation}
while
\begin{equation}\label{eq:norm}
\vec{x}^{\top}\vec{x}=sa^2+c^2+tb^2.
\end{equation}
Let \(U\subseteq \vec{1}^{\perp}\) be the subspace of all such three-block-constant vectors $\vec{x}$. By the variational characterization of $\lambda_2(T)$, we know that
\begin{equation}\label{eq:lambda2-upper-U}
\lambda_2(T)
\le
\min_{0\neq \vec{x}\in U}\frac{\vec{x}^{\top}L_{T}\vec{x}}{\vec{x}^{\top}\vec{x}}.
\end{equation}

Now set
$\vec{y}:=(c,a,b)^{\top}\in \mathbb{R}^3$,
and define
\[
M=
\begin{pmatrix}
d & -1 & -(d-1)\\
-1& 1  & 0\\
-(d-1)&0&d-1
\end{pmatrix},
\qquad
D=\operatorname{diag}(1,s,t).
\]
Then \eqref{eq:energy}--\eqref{eq:norm} become $\vec{x}^{\top}L_{T}\vec{x}=\vec{y}^{\top}M\vec{y}$ and $\vec{x}^{\top}\vec{x}=\vec{y}^{\top}D\vec{y}$, and the orthogonality condition \eqref{eq:orthogonality} is exactly $\vec{y}^{\top}D\vec{1}_3=0$, where $\vec{1}_3:=(1,1,1)^{\top}$. Therefore
\begin{equation}\label{eq:restricted-min}
\min_{0\neq \vec{x}\in U}\frac{\vec{x}^{\top}L_{T}\vec{x}}{\vec{x}^{\top}\vec{x}}
=
\min_{\substack{0\neq \vec{y}\in \mathbb{R}^3\\ \vec{y}^{\top}D\vec{1}_3=0}}
\frac{\vec{y}^{\top}M\vec{y}}{\vec{y}^{\top}D\vec{y}}.
\end{equation}

Because \(D\) is positive definite, we may set $A:=D^{-1/2}MD^{-1/2}$ and $\vec{u}:=D^{1/2}\vec{1}_3$. Then \(A\) is a real symmetric \(3\times 3\) matrix. For every \(\vec{y}\in \mathbb{R}^3\), writing $\vec{z}:=D^{1/2}\vec{y}$ (equivalently, $\vec{y}=D^{-1/2}\vec{z}$), we have
\[
\vec{y}^{\top}M\vec{y}=\vec{z}^{\top}A\vec{z},
\qquad
\vec{y}^{\top}D\vec{y}=\vec{z}^{\top}\vec{z},
\qquad
\vec{y}^{\top}D\vec{1}_3=\vec{z}^{\top}\vec{u}.
\]
Hence \eqref{eq:restricted-min} becomes
\[
\min_{\substack{0\neq \vec{y}\in \mathbb{R}^3\\ \vec{y}^{\top}D\vec{1}_3=0}}
\frac{\vec{y}^{\top}M\vec{y}}{\vec{y}^{\top}D\vec{y}}
=
\min_{\substack{0\neq \vec{z}\in \mathbb{R}^3\\ \vec{z}\perp \vec{u}}}
\frac{\vec{z}^{\top}A\vec{z}}{\vec{z}^{\top}\vec{z}}.
\]

Next, since \(M\vec{1}_3=0\), we obtain
\[
A\vec{u}
=
D^{-1/2}MD^{-1/2}D^{1/2}\vec{1}_3
=
D^{-1/2}M\vec{1}_3
=
0.
\]
So \(\vec{u}\) is an eigenvector of the symmetric matrix \(A\) with
eigenvalue \(0\). 
For \(\vec{z}\in \mathbb R^3\), writing \(\vec{z}=D^{1/2}\vec{y}\), we have
\[
\vec{z}^\top A \vec{z} = \vec{y}^\top M \vec{y} = (a-c)^2+(d-1)(b-c)^2 \ge 0.
\]
Thus \(A\) is positive semidefinite. 
It follows that \(0\) is
the smallest eigenvalue of \(A\).

Since \(A\) is symmetric, there exists an orthonormal basis
of eigenvectors \(\{\vec{e}_1,\vec{e}_2,\vec{e}_3\}\) with
\(
A\vec{e}_i=\mu_i \vec{e}_i
\)
for $i=1,2,3$.
We may assume that
\(0=\mu_1\le \mu_2\le \mu_3,\)
and \(\vec{e}_1=\vec{u}/\|\vec{u}\|\). In particular, the subspace \(\vec{u}^\perp\) is
spanned by \(\vec{e}_2,\vec{e}_3\). By the Rayleigh--Ritz theorem for the restriction of
\(A\) to \(\vec{u}^\perp\),
\[
\min_{\substack{0\neq \vec{z}\in \mathbb{R}^3\\ \vec{z}\perp \vec{u}}}
\frac{\vec{z}^{\top}A\vec{z}}{\vec{z}^{\top}\vec{z}}
=
\mu_2.
\]
Now \(\mu\) is an eigenvalue of \(A=D^{-1/2}MD^{-1/2}\) if and only if
$\det(A-\mu I)=0$,
which is equivalent, because \(D\) is invertible, to \(\det(M-\mu D)=0\). A direct computation gives
\[
\det(M-xD)
=
x\Bigl(
-stx^2+[dst+(d-1)s+t]x-(d-1)n
\Bigr).
\]
Therefore the other two eigenvalues are precisely the two roots of
\[
p_s(x):=
stx^2-[dst+(d-1)s+t]x+(d-1)n.
\]
Let \(\rho(s)\) denote the smaller positive root of \(p_s(x)=0\). Then
\(\rho(s)=\mu_2\), and hence, by \eqref{eq:lambda2-upper-U} and
\eqref{eq:restricted-min},
\begin{equation}\label{eq:lambda2-le-rho}
\lambda_2(T)\le \rho(s).
\end{equation}

\medskip
\noindent
\underline{\textbf{Proof of \(\rho(s)<(d+1)/n\).}}
Define $g(s):=n^2 p_s\left(\frac{d+1}{n}\right)$. A straightforward simplification gives
\[
g(s)=
(d+1)(dn-d-1)s^2
-(d+1)(dn^2-dn+d-3n+1)s
+n(dn^2-dn+d-n^2-n+1).
\]
By \eqref{eq:s-range}, we have
$s\in \left[\frac{n-1}{d},\,\frac n2\right]$.
We claim that $g(s)<0$ always holds.
Since \((d+1)(dn-d-1)>0\), the function \(g(s)\) is convex in \(s\), which attains its maximum at an endpoint.
Hence, to prove the claim, it suffices to evaluate \(g(s)\) at the two endpoints $(n-1)/d$ and $n/2$.
First, we consider
\[
g\left(\frac{n-1}{d}\right)
=
\frac{(d-1)(d-n+1)(dn^2-dn+d-n+1)}{d^2}.
\]
By~\eqref{eq:s-range}, we know that \(d-n+1<0\) and
\(dn^2-dn+d-n+1\ge dn-n+1=(d-1)n+1>0.\)
Therefore
\begin{equation}\label{eq:g-left-neg}
g\left(\frac{n-1}{d}\right)<0.
\end{equation}
Next, we consider 
\begin{equation}\label{eq:g-right}
g\left(\frac n2\right)
=
-\frac n4\,E(d),
\end{equation}
where $E(d)=(n^2-n+2)(d-2)^2+(n^2-6n+8)(d-2)+(2n^2-9n+6)$. 
For \(n\ge 4\), clearly we have \(E(d)>0\), and hence by \eqref{eq:g-right} we know that
\(g\left(\frac n2\right)<0.\)
Combining this with \eqref{eq:g-left-neg}, we have
\[
p_s\left(\frac{d+1}{n}\right)=\frac{g(s)}{n^2}<0.
\]

On the other hand, $p_s(0)=(d-1)n>0$, and the quadratic coefficient of \(p_s\) is \(st>0\). Hence, the smaller positive root \(\rho(s)\) of \(p_s\) must satisfy \(\rho(s)<\frac{d+1}{n}\). Combining this with \eqref{eq:lambda2-le-rho}, we obtain
\begin{equation}\label{eq:lambda2-upper-final}
\lambda_2(T)<\frac{d+1}{n}.
\end{equation}

\medskip
\noindent
\underline{\textbf{A lower bound for \(\lambda_n(T)\) and the conclusion.}}
Consider the principal submatrix \(H\) of \(L_T\) indexed by the vertex set \(\{v\}\cup N(v)\). Then \(H\) is a \((d+1)\times(d+1)\) matrix. Since each neighbour of \(v\) may have additional neighbours outside \(\{v\}\cup N(v)\), we may write
\[
H=L_{S_{d+1}}+\operatorname{diag}(0,\varepsilon_1,\dots,\varepsilon_d)
\]
for some integers \(\varepsilon_1,\dots,\varepsilon_d\ge 0\).
Therefore, by Cauchy interlacing for principal submatrices,
\begin{equation}\label{eq:lambdan-lower}
\lambda_n(T)\ge \lambda_{\max}(H)\ge \lambda_{\max}(L_{S_{d+1}})=d+1.
\end{equation}
Combining \eqref{eq:lambda2-upper-final} with \eqref{eq:lambdan-lower}, we can derive 
\[
\frac{\lambda_2(T)}{\lambda_n(T)}
<
\frac{(d+1)/n}{d+1}
=
\frac1n.
\]
Thus every non-star tree \(T\) satisfies
$\frac{\lambda_2(T)}{\lambda_n(T)}<\frac1n$.
Since equality \(\lambda_2(T)/\lambda_n(T)=1/n\) holds for \(S_n\), we conclude
that for every tree \(T\) on \(n\ge 3\) vertices,
$\frac{\lambda_2(T)}{\lambda_n(T)}\le \frac1n$,
with equality if and only if \(T\cong S_n\).
\end{proof}

\section{From Laplacian Eigenratio to Hamiltonicity via Expansion}\label{sec:hamiltonicity}

In this section, we prove Theorem~\ref{thm:GHmain}, confirming Gu's Hamiltonicity conjecture. Recall that a \emph{Hamilton cycle} in a graph $G$ is a cycle containing all vertices of $G$.

Serving as the principal tool for proving Theorem~\ref{thm:GHmain}, we first follow \cite{HamExpanders} to introduce the definition of $C$-expanders. This concept is intimately tied to spectral graph properties. For a subset $X\subseteq V(G)$, we denote
\[
N(X):=\{v\in V(G)\setminus X:\exists\,u\in X\text{ such that }uv\in E(G)\}.
\]
\begin{definition}[$C$-expander]\label{def:C-expander}
    Let $C>0$ be a constant. An $n$-vertex graph $G$ with $n \ge 3$ is called a \emph{$C$-expander} if it satisfies the following two conditions:
    \begin{itemize}
        \item[(a)]  $|N(X)| \ge C|X|$ for every subset $X \subseteq V(G)$ with $|X| < n/(2C)$;
        \item[(b)]  $e(X, Y) \ge 1$ for any two disjoint subsets $X, Y \subseteq V(G)$ with $|X|, |Y| \ge n/(2C)$.
    \end{itemize}
\end{definition}

Recently, Dragani\'c, Montgomery, Correia, Pokrovskiy and Sudakov proved the following breakthrough result in \cite{HamExpanders}, establishing a bridge between Hamiltonicity and expansion properties.

\begin{theorem}[\cite{HamExpanders}, Theorem 1.4]\label{thm:C-expander}
There exists an absolute constant $C_0>0$ such that for any $C\geq C_0$, every $C$-expander is Hamiltonian.
\end{theorem}

In the next lemma, we show that any graph with a large Laplacian eigenratio must be a $C$-expander.

\begin{lemma}\label{lem:C-expander}
Let $C \ge 2$. Every graph on $n \ge 3$ vertices satisfying $\frac{\lambda_2}{\lambda_n} > 1 - \frac{1}{C}$ is a $C$-expander.
\end{lemma}

Consequently, Theorem~\ref{thm:GHmain} follows immediately from the Hamiltonian criterion for expanders (Theorem~\ref{thm:C-expander}).

\begin{proof}[Proof of Theorem~\ref{thm:GHmain}, assuming Lemma~\ref{lem:C-expander}]
We prove that the theorem holds with
\[
c = 1 - \frac{1}{2\max\{2, C_0\}},
\]
where \(C_0\) is the constant from Theorem~\ref{thm:C-expander}. Assume Lemma~\ref{lem:C-expander} holds. Set \(C := \max\{2, C_0\}\) and \(c := 1 - 1/(2C)\). If \(\frac{\lambda_2}{\lambda_n} \ge c = 1 - \frac{1}{2C}\), then in particular \(\frac{\lambda_2}{\lambda_n} > 1 - \frac{1}{C}\). By Lemma~\ref{lem:C-expander}, \(G\) is a \(C\)-expander. Since \(C \ge C_0\), Theorem~\ref{thm:C-expander} implies that \(G\) is Hamiltonian. Thus Theorem~\ref{thm:GHmain} holds with this choice of \(c\).
\end{proof}

\subsection{Proof of Lemma~\ref{lem:C-expander}}

To prove Lemma~\ref{lem:C-expander}, we use the following well-known eigenvalue separation inequality; see, for instance, \cite[Theorem~3.1]{GH22} or \cite[Lemma~6.1]{Haemers95}.

\begin{theorem}[Haemers~\cite{Haemers95}]\label{Thm:XY}
Let $X$ and $Y$ be two disjoint subsets of an $n$-vertex graph $G$ such that there is no edge between $X$ and $Y$. Denote by $0=\lambda_1\le\cdots\le\lambda_n$ the eigenvalues of the Laplacian matrix of $G$. Then
\[
\frac{|X|\,|Y|}{(n-|X|)(n-|Y|)} \le \left( \frac{\lambda_n - \lambda_2}{\lambda_n + \lambda_2} \right)^{\!2}.
\]
\end{theorem}

With Theorem~\ref{Thm:XY} in hand, we now proceed to the proof of Lemma~\ref{lem:C-expander}.
\begin{proof}[Proof of Lemma~\ref{lem:C-expander}]
    Define $f(r)=\left(\frac{1-r}{1+r}\right)^{\!2}$, which is decreasing on $r\in(0,1]$.
    Let $r_0=\frac{\lambda_2}{\lambda_n} > 1-\frac{1}{C}$.
    Since $C\ge 2$, we have $r_0 > \frac{C-1}{C}$, and therefore
    \begin{equation}\label{eq:f(r)}
        \left(\frac{\lambda_n-\lambda_2}{\lambda_n+\lambda_2}\right)^{\!2}
        =f(r_0)
        < f\!\left(\frac{C-1}{C}\right)
        =\frac{1}{(2C-1)^2}
        \le \frac{C-1}{(C+1)(2C-1)} .
    \end{equation}

   \noindent\textbf{(a).}
Let $X\subseteq V(G)$ with $|X|<\frac{n}{2C}$.
If $X=\varnothing$, then the claim is trivial. Hence assume $X\neq\varnothing$.
Set $Y=V(G)\setminus (X\cup N(X))$.
    Then $X$ and $Y$ are disjoint and $e(X,Y)=0$.
    Write $|X|=xn$, $|Y|=yn$, and $|X\cup N(X)|=zn$, so that $y+z=1$.
    By Theorem~\ref{Thm:XY},
    \[
    \frac{xy}{(1-x)(1-y)}
    = \frac{|X|\,|Y|}{(n-|X|)(n-|Y|)}
    \le \left(\frac{\lambda_n-\lambda_2}{\lambda_n+\lambda_2}\right)^{\!2}
    =f(r_0).
    \]
    Rearranging gives
    $\frac{y}{1-y}\le \frac{1-x}{x}\,f(r_0)$,
    and since $y=1-z$, this yields
    \[
    \frac{1-z}{z}\le \frac{1-x}{x}\,f(r_0),
    \qquad\text{or equivalently}\qquad
    \frac{z}{x}\ge \frac{1}{(1-x)f(r_0)+x}.
    \]
    Because $f(r_0)<\frac{C-1}{(C+1)(2C-1)}$ by~\eqref{eq:f(r)} and $x\le\frac{1}{2C}$, we obtain
    \[
    \frac{z}{x}
    \ge \frac{1}{1-(1-x)\left(1-f(r_0)\right)}
    \ge \frac{1}{1-\bigl(1-\frac{1}{2C}\bigr)
                \left(1-\frac{C-1}{(C+1)(2C-1)}\right)}
    = C+1 .
    \]
    Consequently,
    \[
    \frac{|N(X)|}{|X|}
    = \frac{|X\cup N(X)|-|X|}{|X|}
    = \frac{z-x}{x}
    \ge (C+1)-1 = C .
    \]

    \noindent\textbf{(b).}
    Assume, for contradiction, that there exist disjoint subsets $X,Y\subseteq V(G)$ with
    $|X|,|Y|\ge\frac{n}{2C}$ and $e(X,Y)=0$.
    Set $x=\frac{|X|}{n}$, $y=\frac{|Y|}{n}$; then $x,y\ge\frac{1}{2C}$.
    Applying Theorem~\ref{Thm:XY} again gives
    \[
    \frac{xy}{(1-x)(1-y)}
    \le \left(\frac{\lambda_n-\lambda_2}{\lambda_n+\lambda_2}\right)^{2}
    = f(r_0)
    < \frac{1}{(2C-1)^2},
    \]
    where the last inequality follows from~\eqref{eq:f(r)}.
    But the assumption $x,y\ge\frac{1}{2C}$ implies
    \[
    \frac{xy}{(1-x)(1-y)}
    \ge \frac{(\frac{1}{2C})^2}{\bigl(1-\frac{1}{2C}\bigr)^2}
    = \frac{1}{(2C-1)^2},
    \]
    which contradicts the previous inequality.
    This proves that $e(X,Y)\ge 1$ for any disjoint $X,Y$ with $|X|,|Y|\ge\frac{n}{2C}$.

    Both conditions (a) and (b) of Definition~\ref{def:C-expander} are satisfied. Therefore $G$ is a $C$-expander.
\end{proof}

\section{Concluding remarks}\label{sec:conclusion}

Our main results show that, in the simple-graph setting, Spielman's conjecture exhibits a mixed and rather delicate behavior. It holds for all regular graphs and for every average degree \(d\le 2\), but fails for infinitely many average degrees strictly greater than \(2\). The overall picture beyond this remains unclear, which leads naturally to the following problem.

\begin{problem}
Characterize all rational numbers $d>2$ for which Conjecture~\ref{conj:spielman} holds. 
\end{problem}

\noindent In particular, it would be interesting to understand this for all sufficiently large $d$.

More broadly, our results suggest that the Laplacian eigenratio is a natural graph parameter capturing meaningful structural properties. For instance, Lemma~\ref{lem:C-expander} shows that a large Laplacian eigenratio enforces strong expansion. 
It would be interesting to further investigate its connections with other graph invariants, and to better understand its role in governing global properties of graphs.

\section*{Acknowledgements}
The authors thank Noga Alon and Zilin Jiang for helpful discussions on an earlier version of the manuscript.
This work is supported by National Key Research and Development Program of China 2023YFA1010201, National Natural Science Foundation of China grant 12125106, and Innovation Program for Quantum Science and Technology 2021ZD0302902.

\bigskip

\bigskip

Jie Ma: \href{mailto:jiema@ustc.edu.cn}{\nolinkurl{jiema@ustc.edu.cn}}

\medskip

Quanyu Tang: \href{mailto:tang_quanyu@163.com}{\nolinkurl{tang_quanyu@163.com}}

\medskip

Yuchang Wang: \href{mailto:wyc1999@mail.ustc.edu.cn}{\nolinkurl{wyc1999@mail.ustc.edu.cn}}

\medskip

Zhiheng Zheng: \href{mailto:lhotse@mail.ustc.edu.cn}{\nolinkurl{lhotse@mail.ustc.edu.cn}}

\end{document}